\documentclass[10pt]{amsart}
\usepackage{amsmath,amssymb,mathrsfs,hyperref,color}

\usepackage{subfigure}
\usepackage{float}
\usepackage{times}
\usepackage{tikz}
\usepackage{cases}
\usetikzlibrary{intersections}
\usepackage{paralist}
\usepackage{xcolor}
\usepackage{bbm}

\usetikzlibrary{shadings}
\usetikzlibrary[intersections]
\usetikzlibrary{arrows,shapes}

\makeatletter
\def\setliststart#1{\setcounter{\@listctr}{#1}%
  \addtocounter{\@listctr}{-1}}
\makeatother

\makeatletter
\@addtoreset{figure}{section}
\makeatother

\usepackage{calc}
\newtheorem{The}{Theorem}[section]
\newtheorem{Cor}[The]{Corollary}
\newtheorem{Lem}[The]{Lemma}
\newtheorem{Pro}[The]{Proposition}
\theoremstyle{definition}

\newtheorem{defn}[The]{Definition}

\theoremstyle{remark}
\newtheorem{Rem}[The]{Remark}
\newtheorem{ex}[The]{Example}

\numberwithin{equation}{section}

\newcommand{\R}{\mathbb{R}}

\newcommand{\N}{\mathbb{N}}

\makeatletter
\def\moverlay{\mathpalette\mov@rlay}
\def\mov@rlay#1#2{\leavevmode\vtop{%
   \baselineskip\z@skip \lineskiplimit-\maxdimen
   \ialign{\hfil$\m@th#1##$\hfil\cr#2\crcr}}}
\newcommand{\charfusion}[3][\mathord]{
    #1{\ifx#1\mathop\vphantom{#2}\fi
        \mathpalette\mov@rlay{#2\cr#3}
      }
    \ifx#1\mathop\expandafter\displaylimits\fi}
\makeatother

\newcommand{\cupdot}{\charfusion[\mathbin]{\cup}{\cdot}}

\title[Herglotz' variational principle]{Herglotz' variational principle and Lax-Oleinik evolution}
\author{Piermarco Cannarsa \and Wei Cheng \and Liang Jin \and Kaizhi Wang \and Jun Yan}
\address{Dipartimento di Matematica, Universit\`a di Roma ``Tor Vergata'', Via della Ricerca Scientifica 1, 00133 Roma, Italy}
\email{cannarsa@mat.uniroma2.it}
\address{Department of Mathematics, Nanjing University, Nanjing 210093, China}
\email{chengwei@nju.edu.cn}
\address{Department of Applied Mathematics, Nanjing University of Science and Technology, Nanjing 210094, China}
\email{jl@njust.edu.cn}
\address{School of Mathematical Sciences, Shanghai Jiao Tong University, Shanghai 200240, China}
\email{kzwang@sjtu.edu.cn}
\address{School of Mathematical Sciences, Fudan University and Shanghai Key Laboratory for Contemporary Applied Mathematics, Shanghai 200433, China}
\email{yanjun@fudan.edu.cn}
\date{\today}
\subjclass[2010]{35F21, 49L25, 37J50}
\keywords{Herglotz' variational principle, Hamilton-Jacobi equation, viscosity solution.}
\begin{document}

\begin{abstract}
	We develop an elementary method to give a Lipschitz estimate for the minimizers in the problem of Herglotz' variational principle proposed in \cite{CCWY2018} in the time-dependent case. We deduce Erdmann's condition and the Euler-Lagrange equation separately under different sets of assumptions, by using a generalized du Bois-Reymond lemma. As an application, we obtain a representation formula for the viscosity solution of the Cauchy problem for the Hamilton-Jacobi equation
	\begin{align*}
		D_tu(t,x)+H(t,x,D_xu(t,x),u(t,x))=0
	\end{align*}
	and study the related Lax-Oleinik evolution.
\end{abstract}

\maketitle


\section{Introduction}

\subsection{Introduction}

A basic problem of calculus of variations is to minimize the action functional
\begin{align*}
	\int^b_aL(s,\xi(s),\dot{\xi}(s))\ ds
\end{align*}
over the set of absolutely continuous curves $\xi$ connecting two points $x,y\in\R^n$. It has been studied now for almost three hundred years. Beyond the issue of the existence of minimizers, much of the attention in the calculus of variations has been devoted to necessary conditions for optimality. Another essential point of the analysis is the Lipschitz regularity of minimizers. This property has many applications, for instance to Euler-Lagrange equations, where it can be used to exclude the Lavrentiev phenomenon (see, for instance, \cite{Buttazzo_Belloni1995} for a survey on this topic). The Lipschitz regularity of minimizers is the subject of an extensive literature (see, for instance, \cite{Clarke_Vinter1985_1,Ambrosio_Ascenzi_Buttazzo1989,Sychev1992,Dal_Maso_Frankowska_2003,Clarke2007,Buttazzo_Giaquinta_Hildebrandt_book,Clarke_book2013}).

This paper is devoted to the {\em generalized variational principle} proposed by Gustav Herglotz in 1930 (\cite{Herglotz1930,Herglotz1979}).  Such a result generalizes classical variational principles by defining a functional whose extrema are sought by a differential equation.

More precisely, let $L\in C^2(\R\times\R^n\times\R^n\times\R,\R)$ and $\xi:[a,b]\to\R^n$ be any piecewise $C^1$ curve. The functional $u_{\xi}$ is defined in an implicit way by the ordinary differential equation
\begin{equation}\label{eq:Cara_intro}
	\dot{u}_{\xi}(s)=L(s,\xi(s),\dot{\xi}(s),u_{\xi}(s)),\quad s\in[a,b],
\end{equation}
with $u_{\xi}(a)=u\in\R$, for $b>a$. The so-called Herglotz' variational principle is to seek an extremal $\xi$ of the functional
\begin{align*}
	u[\xi]:=u_{\xi}(b)-u=\int^b_aL(s,\xi(s),\dot{\xi}(s),u_{\xi}(s))\ ds,
\end{align*}
where $u_{\xi}$ is determined by \eqref{eq:Cara_intro}. We call $\xi$ is an extremal of $u[\xi]$ if $\frac d{d\varepsilon}u[\xi+\varepsilon\eta]=0$ for arbitray piecewise $C^1$ curve $\eta$ such that $\eta(a)=\eta(b)=0$.  Herglotz' variational principe gurantees that any $C^2$ extremal of the functional $u[\xi]$ must satisfy the so-called Herglotz equation
\begin{equation}\label{eq:Hergotz_intro}
	\frac d{ds}L_v=L_x+L_uL_v.
\end{equation}
Herglotz reached the idea of the generalized variational principle through his work on contact transformations and their connections with Hamiltonian systems and Poisson brackets. The reader can find more information on the problem and its rather wide connections in \cite{CCWY2018} (see also \cite{Guenther_Guenther_Gottsch1995,Giaquinta_Hildebrandt_book_I,Giaquinta_Hildebrandt_book_II}) and the references therein. However, to our knowledge, there is no rigorous approach to this problem in a modern setting including the existence and regularity results.

\subsection{Assumptions on $L$}

Now, we impose our assumptions on the Lagrangian $L$. Let $L=L(t,x,v,r):\R\times\R^n\times\R^n\times\R\rightarrow\R$ be a function of class $C^1$ such that the following standing assumptions are satisfied:
\begin{enumerate}[(L1)]
  \item $L(t,x,\cdot,r)$ is strictly convex for all $(t,x,r)\in \R\times\R^n\times\R$.
  \item There exist two superlinear functions $\overline{\theta}_0,\theta_0:[0,+\infty)\to[0,+\infty)$ and two $L^{\infty}_{\rm loc}$-functions $c_0,c_1:\R\to[0,+\infty)$, such that
  \begin{align*}
  	\overline{\theta}_0(|v|)+c_1(t)\geqslant L(t,x,v,0)\geqslant\theta_0(|v|)-c_0(t),\quad (t,x,v)\in\R\times\R^n\times\R^n\times\R.
  \end{align*}
  \item There exists an $L^{\infty}_{\rm loc}$-function $K:\R\to[0,+\infty)$ such that
  \begin{align*}
  	|L_r(t,x,v,r)|\leqslant K(t),\quad (t,x,v,r)\in\R\times\R^n\times\R^n\times\R.
  \end{align*}
  \item There exists two $L^{\infty}_{\rm loc}$-functions $C_1,C_2:\R\to[0,\infty)$ such that
  \begin{align*}
  	|L_t(t,x,v,r)|\leqslant C_1(t)+C_2(t)L(t,x,v,r),\quad (t,x,v,r)\in\R\times\R^n\times\R^n\times\R.
  \end{align*}
\end{enumerate}

There are various conditions that may replace (L4). We will mainly focus on the following substitution of (L4):
\begin{enumerate}[(L1')]\setliststart{4}
	\item There exist two $L^{\infty}_{\rm loc}$-functions $C_1,C_2:\R\to[0,\infty)$ such that for all $(t,x,v,r)\in\R\times\R^n\times\R^n\times\R$
  \begin{align*}
  	\max\{|L_x(t,x,v,r)|,|L_v(t,x,v,r)|\}\leqslant C_1(t)+C_2(t)L(t,x,v,r).
  \end{align*}
\end{enumerate}

\begin{Rem}\label{rem:constants}
If $a<b$ are fixed and $L$ is restricted on $[a,b]\times\R^n\times\R^n\times\R$, then the $L^{\infty}_{\rm loc}$-functions $c_0(t),c_1(t),K(t),C_1(t),C_2(t)$ appear in our assumptions on $L$ can be chosen as constants, say $c_0,c_1,K,C_1,C_2$ (we also set $c_1=0$ for convenience). In fact, we can also assume $C_1\in L^1$ in condition (L4) and (L4') respectively.
\end{Rem}


\subsection{Herglotz' variational principle}

Fix $x,y\in\R^n$, $a<b$ and $u\in\R$. Set
\begin{align*}
	\Gamma^{a,b}_{x,y}=\{\xi\in W^{1,1}([a,b],\R^n): \xi(a)=x,\ \xi(b)=y\}.
\end{align*}
For any given $\xi\in\Gamma^{a,b}_{x,y}$, we consider the Carath\'eodory equation
\begin{equation}\label{eq:app_caratheodory_L}
	\begin{cases}
		\dot{u}_{\xi}(s)=L(s,\xi(s),\dot{\xi}(s),u_{\xi}(s)),\quad a.e.\ s\in[a,b],&\\
		u_{\xi}(a)=u.&
	\end{cases}
\end{equation}
We define the action functional
\begin{equation}\label{eq:app_fundamental_solution}
	J(\xi):=\int^b_aL(s,\xi(s),\dot{\xi}(s),u_{\xi}(s))\ ds,
\end{equation}
where $\xi\in\Gamma^{a,b}_{x,y}$ and $u_{\xi}$ is uniquely determined by \eqref{eq:app_caratheodory_L} by Proposition \ref{caratheodory}. Our purpose is to minimize $J(\xi)$ over
\begin{align*}
	\mathcal{A}=\mathcal{A}^{a,b,u}_{x,y}=\{\xi\in\Gamma^{a,b}_{x,y}: \text{\eqref{eq:app_caratheodory_L} admits an absolutely continuous solution $u_{\xi}$}\}.
\end{align*}
Notice that $\mathcal{A}\not=\varnothing$ because it contains all piecewise $C^1$ curves connecting $x$ to $y$. It is not hard to check that, for each $r\in\R$,
\begin{align*}
	\mathcal{A}=\mathcal{A}':=\{\xi\in\Gamma^{a,b}_{x,y}: \text{$s\mapsto L(s,\xi(s),\dot{\xi}(s),r)$ belongs to $L^1([a,b])$}\}.
\end{align*}
In fact, what we are studying is a variational problem under a very special non-holonomic constraint. The readers can refer to, for instance, \cite{Giaquinta_Hildebrandt_book_II}. Our work is essentially motivated by the recent works \cite{Wang_Wang_Yan2017,CCWY2018,Wang_Yan2019}.

\begin{Pro}\label{existence_intro}
Fix $x,y\in\R^n$, $b>a$ and $u\in\R$. Under conditions \mbox{\rm (L1)-(L3)}, the functional
	\begin{align*}
		\mathcal{A}\ni\xi\mapsto J(\xi)=\int^b_aL(s,\xi(s),\dot{\xi}(s),u_{\xi}(s))\ ds,
	\end{align*}
	where $u_{\xi}$ is determined by \eqref{eq:app_caratheodory_L}, admits a minimizer.	
\end{Pro}

The proof of Proposition \ref{existence_intro} is given in Appendix \ref{sec:existence} (see \cite{CCWY2018} for the time-independent case).

\subsection{Erdmann condition and Herglotz equation}
From the technical point of view, this is the main part of this paper. Since the action functional $J$ is essentially defined in an implicit way, to our knowledge, all the methods in the standard references such as \cite{Clarke_Vinter1985_1}, \cite{Ambrosio_Ascenzi_Buttazzo1989} or \cite{Dal_Maso_Frankowska_2003} can not be applied directly. In the previous paper \cite{CCWY2018}, due to summability issues, we  solved this problem under restrictive growth conditions on $L$ for the autonomous case. In this paper, appealing to additional technical tools, we solve this problem as follows:
\begin{enumerate}[(1)]
  \item We improve the classical du Bois-Reymond lemma in the calculus of variations proving that such a lemma holds even if the test functions are selected in a restricted space. More precisely, suppose $f,g\in L^1([a,b])$, $\delta\in L^{\infty}([a,b])$ and $\delta(s)>0$ for almost all $s\in[a,b]$. Set the family of test functions as $\Omega=\{\beta\in L^{\infty}([a,b]): \int^b_a\beta(s)\ ds=0, |\beta|\leqslant\delta, a.e.\}$. We will show, if
  \begin{align*}
  	\int^b_af(s)b_{\beta}(s)+g(s)\beta(s)\ ds=0,\quad \beta\in\Omega,
  \end{align*}
  where $b_{\beta}(s):=\int^s_a\beta(r)\ dr$ for $\beta\in\Omega$, then there exists a continuous representative $\tilde{g}$ of $g$ such that $\tilde{g}$ is absolutely continuous on $[a,b]$ and $\tilde{g}'(s)=f(s)$ for almost all $s\in[a,b]$.
  \item We have to deal with the problem under various sets of conditions separately. If condition (L1)-(L3) together with (L4) are satisfied, we will adopt the method of \cite{Ambrosio_Ascenzi_Buttazzo1989} based on reparameterization. Without loss of generality we set $[a,b]=[0,t]$ for $t>0$. For any measurable function $\alpha:[0,t]\to[1/2,3/2]$ satisfying $\int^t_0\alpha(s)\ ds=t$, we define $\tau(s)=\int^s_0\alpha(r)\ dr$ for $s\in [0,t]$. Note that $\tau:[0,t]\to[0,t]$ is a bi-Lipschitz map.

  Now, let $\xi\in\Gamma^{0,t}_{x,y}$ be a minimizer of $J$,  and $\alpha\in\Omega$ as above. We define the reparameterization $\eta$ of $\xi$ by $\eta(\tau)=\xi(s(\tau))$ where $s(\tau)$ is the inverse of $\tau(s)$. It follows that $\dot{\eta}(\tau)=\dot{\xi}(s(\tau))/\alpha(s(\tau))$. Let $u_{\eta}$ be the unique solution of \eqref{eq:app_caratheodory_L} with initial condition $u_{\eta}(0)=u$. Then we have that
  \begin{align*}
	J(\xi)\leqslant\, J(\eta)=&\int^t_0L(\tau,\eta(\tau),\dot{\eta}(\tau),u_{\eta}(\tau))\ d\tau\\
	=&\,\int^t_0L(\tau(s),\xi(s),\dot{\xi}(s)/\alpha(s),u_{\xi,\alpha}(s))\alpha(s)\ ds
  \end{align*}
  where $u_{\xi,\alpha}$ solves
  \begin{align*}
	\dot{u}_{\xi,\alpha}(s)=L(\tau(s),\xi(s),u_{\xi,\alpha}(s),\dot{\xi}(s)/\alpha(s))\alpha(s),\quad u_{\xi,\alpha}(0)=u.
  \end{align*}
  Define the functional $\Lambda:\Omega\to\R$ by
  \begin{align*}
	\Lambda(\alpha)=u_{\xi,\alpha}(t)
  \end{align*}
  with $u_{\xi,\alpha}$ as above. We write $\alpha=1+\beta$. We should verify
  \begin{align*}
	0=\frac d{d\varepsilon}\Lambda(1+\varepsilon\beta)\vert_{\varepsilon=0}=\int^b_a\left\{E\cdot\beta-e^{-\int^s_aL_udr}\,L_t(s,\xi,\dot{\xi},u_{\xi})\cdot b_{\beta}\right\}\ ds,
\end{align*}
where $E(s)=e^{-\int^s_aL_udr}\cdot\left\{L_v(s,\xi(s),\dot{\xi}(s),u_{\xi}(s))\cdot\dot{\xi}(s)-L(s,\xi(s),\dot{\xi}(s),u_{\xi}(s))\right\}$, to obtain the Erdmann condition. From technical point of view, we need validate the convergence by using Lebesgue's theorem. That means we need check the required  summability issues. This problem is solved by using our conditions (L1), (L3) and (L4) and the restriction of $\beta\in\Omega$. Now, invoking our generalized du Bois-Reymond lemma, we obtain the Erdmann condition
  \begin{equation}\label{eq:Erdmann_introduction}
  	\frac d{ds}\ E(s)=-e^{-\int^s_aL_udr}L_t(s),\quad a.e.\ s\in[0,t].
  \end{equation}
  \item If conditions (L1)-(L3) together with (L4') are satisfied, we use the standard variation $\xi_{\varepsilon}=\xi+\varepsilon\eta$. Also for the summability difficulty, we restrict our $\eta\in\Omega$. One can deduce the Herglotz equation \eqref{eq:Hergotz_intro} on $[a,b]$ almost everywhere by using the generalized du Bois-Reymond lemma.
\end{enumerate}


If $L$ is of class $C^2$, then any minimizer $\xi$ of \eqref{eq:app_fundamental_solution} is as smooth as $L$ and $\xi$ satisfies Herglotz equation \eqref{eq:Hergotz_intro} on $[a,b]$ where $u_{\xi}$ is of class $C^2$ and satisfies Carath\'eodory ODE \eqref{eq:app_caratheodory_L}. Let $H$ be the associated Hamiltonian defined by
\begin{align*}
	H(t,x,p,r)=\sup_{v\in\R^n}\{p\cdot v-L(t,x,v,r)\},\quad t\in\R, (x,v)\in\R^n\times\R^n, r\in\R.
\end{align*}
Then $H$ is also of class $C^2$ and satisfies certain standard conditions.

Set $p(s)=L_v(s,\xi(s),\dot{\xi}(s),u_{\xi}(s))$. Then the arc $(\xi,p,u_{\xi})$ satisfies the following Lie equation
\begin{equation}\label{eq:Lie}
  	\begin{cases}
  		\dot{\xi}=H_p(s,\xi,u_{\xi},p),\\
  		\dot{p}=-H_x(s,\xi,u_{\xi},p)-H_u(s,\xi,u_{\xi},p)p,\qquad s\in[a,b],\\
  		\dot{u}_{\xi}=p\cdot\dot{\xi}-H(s,\xi,u_{\xi},p).
  	\end{cases}
\end{equation}
Equation \eqref{eq:Lie} is a special kind of contact system. The readers can also recognize \eqref{eq:Lie} as the system of characteristics. This system is widely studied in mathematics (see, for instance, \cite{Arnold_book_1989,Evans_book_2010} for general information and \cite{Su_Wang_Yan2016,Wang_Wang_Yan2019_1,Wang_Wang_Yan2019_2,Maro_Sorrentino2017,Zhao_Cheng2018,CCIZ2019} especially on connections to Aubry-Mather theory and Hamilton-Jacobi equations), mechanics and mathematical physics (see, for instance, \cite{Bravetti2017,Bravetti_Cruz_Tapias2017,Liu_Torres_Wang2018} and  \cite{Nose1984,Hoover1985,Posch_Hoover_Vesely1986,Legoll_Luskin_Moeckel2007} for Nos\'e-Hoover dynamics).

\subsection{Hamilton-Jacobi equations of contact type}

As an application, this paper establishes a connection between Herglotz' variational problem and the Hamilton-Jacobi equation
\begin{equation}\label{eq:HJe_t}\tag{HJ}
\left\{
\begin{split}
	D_tu(t,x)+H(t,x,D_xu(t,x),u(t,x))=&\,0\\
		u(0,x)=&\,\phi(x)
\end{split}
\right.\quad x\in \R^n, t>0,
\end{equation}
the solution of which is given by the related Lax-Oleinik evolution.

We suppose $L$ is of class $C^1$ satisfying conditions (L1)-(L3) together with (L4) (resp. (L4'), with $H$ being the associated Hamiltonian. Fix $x,y\in\R^n$, $t_2>t_1$ and $u\in\R$. We define
\begin{align*}
	h_L(t_1,t_2,x,y,u):=\inf_{\xi\in\mathcal{A}^{t_1,t_2}_{x,y}}\int^{t_2}_{t_1}L(s,\xi(s),\dot{\xi}(s),u_{\xi}(s))\ ds,
\end{align*}
where $u_{\xi}$ is determined by the associated Carath\'eodory equation. The function $h_L$ is called the fundamental solution of \eqref{eq:HJe_t}.

To study \eqref{eq:HJe_t} for a wider class of the initial data. We suppose that $\phi$ is a real-valued function on $\R^n$ which is lower semi-continuous and $(\kappa_1,\kappa_2)$-Lipschitz in the large (see Definition \ref{defn:Lip_in_large}). The main result is that
\begin{align*}
	u(t,x)=\inf_{y\in\R^n}\{\phi(y)+h_L(0,t,y,x,\phi(y))\}
\end{align*}
is finite-valued and it is a viscosity solution of \eqref{eq:HJe_t}. We also introduce the Lax-Oleinik evolution in this context and discuss the related dynamic programming principle. A systematic approach to this problem from Lagrangian formalism will be our task in the future.


\begin{ex}\label{example:duffing}
Let $V$ be a smooth real-valued function on $\R^n\times\R$, $\lambda\in\R$ and let
\begin{align*}
	L(s,x,v,r)=L_0(s,x,v)-\lambda r,
\end{align*}
where $L_0=\frac 12|v|^2-V(x,t)$. Then the associated Herglotz equation, i.e.,
\begin{align*}
	\ddot{x}+\lambda\dot{x}+\nabla_xV(x,t)=0,
\end{align*}
is a Duffing-type equation, which is rather widely studied in many fields such as mechanics, nonlinear physics and engineering (see, for instance, \cite{Moser_book1973}). Recall that the associated Hamiltonian has the form $H=H_0(s,x,p)+\lambda r$ where $H_0$ is the Fenchel-Legendre dual of $L_0$. This model is also closely related to  discounted Hamilton-Jacobi equations in PDE and calculus of variations and optimal control \cite{Arisawa1997,Arisawa1998,Cannarsa_Quincampoix2015,Gomes2008,DFIZ2016,Ishii_Mitake_Tran2017_1,Ishii_Mitake_Tran2017_2,Zhao_Cheng2018,CCIZ2019}.
\end{ex}

The paper is organized as follows: In Section 2, we give a detailed proof of the Erdmann condition and Herglotz equation based on our generalized du Bois-Reymond lemma under various kind of conditions. Then we obtain the expected Lipschitz estimates. In Section 3, we apply Herglotz' variational principle to Hamilton-Jacobi equation \eqref{eq:HJe_t}. We have three appendices. In Appendix A, We collect useful material from analysis and differential equations. The Main part of Appendix B is composed of the details of the proofs of a Tonelli-like existence result and some necessary a priori estimates. In Appendix C, we explain how to move Herglotz' variational principle to manifolds.

\medskip

\noindent\textbf{Acknowledgement} This work is partly supported by National Natural Scientific Foundation of China (Grant No.11790272, No.11871267, No.11631006, and No.11771283), and the National Group for Mathematical Analysis, Probability and Applications (GNAMPA) of the Italian Istituto Nazionale di Alta Matematica ``Francesco Severi''. The authors acknowledge the MIUR support from Excellence Department Project awarded to the Department of Mathematics, University of Rome Tor Vergata, CUP E83C18000100006. The authors are grateful to Qinbo Chen, Cui Chen, Jiahui Hong, Shengqing Hu and Kai Zhao for helpful discussions.

\section{Necessary conditions and Lipschitz estimates}

The main purpose of this section is to give a Lipschitz estimate of any minimizer $\xi$ of \eqref{eq:app_fundamental_solution} and to derive some necessary conditions such as the generalized Euler-Lagrange equation (Herglotz equation) and Erdmann condition. Unlike the autonomous case studied in \cite{CCWY2018}, we will deal with the problem under various kind of conditions. It is worthing noting that
\begin{enumerate}[--]
	\item one can deduce the Erdmann condition for the ``energy function'' directly under the conditions (L1)-(L3) together with (L4);
	\item one can also deduce the Herglotz equation directly under the conditions (L1)-(L3) together with (L4')
\end{enumerate}
A key tool is the following lemma of du Bois-Reymond type (see Theorem \ref{du_Bois_Reymond1}). By using such a result, one can get the required Lipschitz estimate after having derived either Erdmann condition or Herglotz equation.

\subsection{A generalized du Bois-Reymond lemma}

\begin{The}[du Bois-Reymond lemma]\label{du_Bois_Reymond1}
Suppose $f,g\in L^1([a,b])$, $\delta\in L^{\infty}([a,b])$ and $\delta(s)>0$ for almost all $s\in[a,b]$. Set
\begin{align*}
	\Omega=\{\beta\in L^{\infty}([a,b]): \int^b_a\beta(s)\ ds=0, |\beta|\leqslant\delta, a.e.\}.
\end{align*}
If
\begin{equation}\label{eq:DuBois-Reymond}
	\int^b_af(s)b_{\beta}(s)+g(s)\beta(s)\ ds=0,\quad \beta\in\Omega,
\end{equation}
where $b_{\beta}(s):=\int^s_a\beta(r)\ dr$ for $\beta\in\Omega$, then there exists a continuous representative $\tilde{g}$ of $g$ such that $\tilde{g}$ is absolutely continuous on $[a,b]$ and $\tilde{g}'(s)=f(s)$ for almost all $s\in[a,b]$.
\end{The}

\begin{proof}
Without loss of generality, we suppose that $\delta\in L^{\infty}([a,b])$ and $\delta(s)>0$ for all $s\in[a,b]$.

Suppose first that $f\equiv0$. Let $\beta\in L^{\infty}([a,b])$, $\|\beta\|_{\infty}\leqslant1$, and $\int^b_a\beta\ ds=0$. Set $A^+=\{\beta\geqslant 0\}$, $A^-=\{\beta<0\}$.

For any $n\geqslant1$, by Lusin's theorem, there exists a compact set $E_n\subset[a,b]$ such that $|E_n|>(b-a)-1/n$ and the restriction of $\delta$ on $E_n$ is continuous. Thus,
\begin{equation}\label{eq:N_n}
	N_n:=\min\{\delta(s): s\in E_n\}>0.
\end{equation}
Set $A^{\pm}_n=A^{\pm}\cap E_n$. Consider the measure $\mu(E)=\int_E|\beta|\ ds$ which is absolutely continuous with respect to Lebesgue measure. We remark that $g\in L^1([a,b],\mu)$. Therefore, for any $\varepsilon>0$ there exists $\sigma_{\varepsilon}>0$ such that for any measurable subset $E\subset[a,b]$ we have that
\begin{equation}\label{eq:AC_1}
	\int_E|\beta|\ ds<\sigma_{\varepsilon}\quad\text{implies}\quad\int_E|g\beta|\ ds=\int_E|g|\ d\mu<\varepsilon/2.
\end{equation}

Fix $\varepsilon>0$ and let $\sigma=\sigma_{\varepsilon}$. Without loss of generality, we suppose that $\int_{A^+_n}|\beta|\ ds>\int_{A^-_n}|\beta|\ ds$\footnote{If the two integrals are equal, we go directly to $\beta_n$ below.}. The other possibility that $\int_{A^+_n}|\beta|\ ds<\int_{A^-_n}|\beta|\ ds$ can be dealt with in a similar way. Then, by \eqref{eq:AC_1}, there exists $n_{\varepsilon}\in\N$ such that for all $n\geqslant n_{\varepsilon}$ we have that
\begin{align*}
	0<\int_{A^+_n}|\beta|\ ds-\int_{A^-_n}|\beta|\ ds<\sigma,
\end{align*}
and
\begin{equation}\label{eq:AC_2}
	\int_{[a,b]\setminus E_n}|g|\ ds<\varepsilon/2.
\end{equation}
Now, define $\psi(s)=\int^s_a|\beta|\cdot\mathbbm{1}_{A^+_n}\ dr$. Then $\psi(b)>\int_{A^-_n}|\beta|\ ds$ and $\psi(a)=0$. Taking $s_n=\sup\{s\in[a,b]:\psi(s)<\int_{A^-_n}|\beta|\ ds\}$, we have that
\begin{align*}
	\int^{s_n}_a|\beta(r)|\cdot\mathbbm{1}_{A^+_n}(r)\ dr=\int_{A^-_n}|\beta(s)|\ ds,
\end{align*}
and, for $n\geqslant n_{\varepsilon}$ we have that
\begin{equation}\label{eq:AC_3}
	\int^b_{s_n}|\beta|\cdot\mathbbm{1}_{A^+_n}\ dr=\int_{A^+_n}|\beta|\ ds-\int_{A^-_n}|\beta|\ ds<\sigma.
\end{equation}

Define
\begin{align*}
	\beta_n(s)=
	\begin{cases}
		\beta(s),&s\in J_n:=(A^+_n\cap[a,s_n])\cupdot A^-_n;\\
		0,&\text{otherwise,}
	\end{cases}
\end{align*}
where $\cupdot$ stands for the union of two disjoint set. Then $N_n\beta_n\in\Omega$ with $N_n$ defined in \eqref{eq:N_n}. Next, suppose $f\equiv0$. Then, in light of \eqref{eq:DuBois-Reymond}, we have that
\begin{equation}\label{eq:DuBois-Reymond2}
	\int^b_ag\beta_n\ ds=0.
\end{equation}
Notice that $E_n=J_n\cupdot(A^+_n\cap(s_n,b])$, or equivalently, $[a,b]\setminus J_n=([a,b]\setminus E_n)\cup(A^+_n\cap(s_n,b])$. Therefore, invoking \eqref{eq:DuBois-Reymond2}, \eqref{eq:AC_2}, \eqref{eq:AC_3} and \eqref{eq:AC_1} and recalling that the integral of $g\beta_n$ vanishes by our assumption, we conclude that for $n\geqslant n_{\varepsilon}$
\begin{align*}
	\left|\int^b_ag\beta\ ds\right|\leqslant&\,\left|\int_{J_n}g\beta\ ds\right|+\left|\int_{[a,b]\setminus J_n }g\beta\ ds\right|\leqslant\left|\int^b_ag\beta_n\ ds\right|+\int_{[a,b]\setminus J_n }|g\beta|\ ds\\
	\leqslant&\,\int_{[a,b]\setminus E_n }|g\beta|\ ds+\int^b_{s_n}|g\beta|\cdot\mathbbm{1}_{A^+_n}\ ds\\
	<&\,\varepsilon.
\end{align*}

Since $\varepsilon$ is arbitrary we conclude that \eqref{eq:DuBois-Reymond}, in the case of $f\equiv0$, holds for any $\beta\in L^{\infty}([a,b])$ such that $\|\beta\|_{\infty}\leqslant1$ and $\int^b_a\beta\ ds=0$. Now, our conclusion is a consequence of the standard du Bois-Reymond lemma (see, for instance, \cite[Lemma 6.1.1]{Cannarsa_Sinestrari_book}).

Finally, to treat the case of $f\not\equiv0$, let $F(s)=\int^s_af(r)\ dr$. Then $F'=f$ almost everywhere on $[a,b]$. Notice that $b_{\beta}(a)=b_{\beta}(b)=0$ for any $\beta\in\Omega$. Then, by \eqref{eq:DuBois-Reymond2}, for any $\beta\in\Omega$ we obtain that
\begin{align*}
	0=&\,\int^b_af(s)b_{\beta}(s)+g(s)\beta(s)\ ds\\
	=&\,\int^b_a\frac d{ds}(F(s)b_{\beta}(s))\ ds+\int^b_a(g(s)-F(s))\beta(s)\ ds\\
	=&\,\int^b_a(g(s)-F(s))\beta(s)\ ds.
\end{align*}
By the first step, we conclude $g-F$  is a.e. equal to some constant $c_0$. So, $\tilde{g}=F+c_0$.
\end{proof}

\subsection{Erdmann condition}

\begin{The}[Erdmann condition]\label{Erdmann_condition}
	Suppose \mbox{\rm (L1)-(L4)} are satisfied. Let $\xi\in\Gamma^{a,b}_{x,y}$ be a minimizer of \eqref{eq:app_fundamental_solution} with $u_{\xi}$ determined by \eqref{eq:app_caratheodory_L}. Set
	\begin{align*}
		\int^s_aL_udr=\int^s_aL_u(r,\xi(r),\dot{\xi}(r),u_{\xi}(r))dr
	\end{align*}
	and define
	\begin{align*}
		E(s):=&\,e^{-\int^s_aL_udr}\cdot\left\{L_v(s,\xi(s),\dot{\xi}(s),u_{\xi}(s))\cdot\dot{\xi}(s)-L(s,\xi(s),\dot{\xi}(s),u_{\xi}(s))\right\}\\
		=&\,e^{-\int^s_aL_udr}\cdot E_0(s)
	\end{align*}
	for almost all $s\in[a,b]$. Then $E$ has a continuous representation $\bar{E}$ such that $\bar{E}$ is absolutely continuous on $[a,b]$ and
	\begin{equation}\label{eq:Erdmann_condition}
		\frac d{ds}\ \bar{E}(s)=-e^{-\int^s_aL_udr}L_t(s)
	\end{equation}
	for almost all $s\in[a,b]$, where $L_t(s)=L_t(s,\xi(s),\dot{\xi}(s),u_{\xi}(s))$.
\end{The}

\begin{proof}
We divide the proof into several steps. Without loss of generality, we suppose the time interval is $[0,t]$ with $t=b-a$.

\medskip

\noindent\textbf{Step I: Reparameterization.} We follow the approach from \cite{Clarke2007}. For any measurable function $\alpha:[0,t]\to[1/2,3/2]$ satisfying $\int^t_0\alpha(s)\ ds=t$ (the set of all such functions $\alpha$ is denoted by $\Omega$), we define
\begin{align*}
	\tau(s)=\int^s_0\alpha(r)\ dr,\quad s\in [0,t].
\end{align*}
Note that $\tau:[0,t]\to[0,t]$ is a bi-Lipschitz map and its inverse $s(\tau)$ satisfies
$$
s'(\tau)=\frac 1{\alpha(s(\tau))},\quad a.e.\ \tau\in[0,t].
$$
Now, given $\xi\in\Gamma^{0,t}_{x,y}$  as above and $\alpha\in\Omega$, define the reparameterization $\eta$ of $\xi$ by $\eta(\tau)=\xi(s(\tau))$. It follows that $\dot{\eta}(\tau)=\dot{\xi}(s(\tau))/\alpha(s(\tau))$. Let $u_{\eta}$ be the unique solution of \eqref{eq:app_caratheodory_L} with initial condition $u_{\eta}(0)=u$. Then we have that
\begin{align*}
	J(\xi)\leqslant\, J(\eta)=&\int^t_0L(\tau,\eta(\tau),\dot{\eta}(\tau),u_{\eta}(\tau))\ d\tau\\
	=&\,\int^t_0L(\tau(s),\xi(s),\dot{\xi}(s)/\alpha(s),u_{\xi,\alpha}(s))\alpha(s)\ ds
\end{align*}
where $u_{\xi,\alpha}$ solves
\begin{equation}\label{eq:u_alpha}
	\dot{u}_{\xi,\alpha}(s)=L(\tau(s),\xi(s),u_{\xi,\alpha}(s),\dot{\xi}(s)/\alpha(s))\alpha(s),\quad u_{\xi,\alpha}(0)=u.
\end{equation}
By a direct calculation, for all $\alpha\in\Omega$ and almost all $s\in[0,t]$, we obtain
\begin{align*}
	\dot{u}_{\xi,\alpha}-\dot{u}_{\xi}=&\,L(\tau,\xi,\dot{\xi}/\alpha,u_{\xi,\alpha})\alpha-L(s,\xi,\dot{\xi},u_{\xi})\\
	=&\,L(\tau,\xi,\dot{\xi}/\alpha,u_{\xi,\alpha})\alpha-L(\tau,\xi,\dot{\xi}/\alpha,u_{\xi})\alpha\\
	&\,+L(\tau,\xi,\dot{\xi}/\alpha,u_{\xi})\alpha-L(s,\xi,\dot{\xi},u_{\xi})\\
	=&\,\widehat{L_u^{\alpha}}\,(u_{\xi,\alpha}-u_{\xi})+(L(\tau,\xi,\dot{\xi}/\alpha,u_{\xi})\alpha-L(s,\xi,\dot{\xi},u_{\xi}))
\end{align*}
and $u_{\xi,\alpha}(0)-u_{\xi}(0)=0$, where
\begin{align*}
	\widehat{L_u^{\alpha}}(s)=\int^1_0L_u\big(\tau(s),\xi(s),\dot{\xi}(s)/\alpha(s),u_{\xi}(s)+\lambda(u_{\xi,\alpha}(s)-u_{\xi}(s))\big)\alpha(s)\ d\lambda.
\end{align*}
By solving the Carath\'eodory equation above, we conclude that
\begin{equation}\label{eq_implicit}
	u_{\xi,\alpha}(s)-u_{\xi}(s)=\int^s_0e^{\int^s_{\sigma}\widehat{L_u^{\alpha}}dr}(L(\tau,\xi,\dot{\xi}/\alpha,u_{\xi})\alpha-L(\sigma,\xi,\dot{\xi},u_{\xi}))\ d\sigma
\end{equation}
and $u_{\xi,\alpha}(t)-u_{\xi}(t)\geqslant0$ for all $\alpha\in\Omega$.

\medskip

\noindent\textbf{Step II: Summability after reparameterization.} For $\alpha\in[1/2,3/2]$ we define
\begin{align*}
	\Phi_1(s,\alpha):=L(s,\xi(s),\dot{\xi}(s)/\alpha,u_{\xi}(s))\alpha-L(s,\xi(s),\dot{\xi}(s),u_{\xi}(s)).
\end{align*}
For almost all $s$, by continuity, there exists $\delta_1(s)\in(0,1/2]$ such that
\begin{align*}
	-1\leqslant\Phi_1(s,\alpha)-\Phi_1(s,1)\leqslant 1,\quad \forall\alpha\in[1-\delta_1(s),1+\delta_1(s)].
\end{align*}
We define a set-valued map $G:[0,t]\rightrightarrows\R$ by
\begin{align*}
	[0,t]\ni s\mapsto G(s)=\{\delta>0: \Phi_1(s,[1-\delta,1+\delta])\subset\Phi_1(s,1)+[-1,1]\},
\end{align*}
and for each $k\in\N$ a set-valued map $G_k:[0,t]\rightrightarrows\R$ by
\begin{align*}
	\text{dom}\,(G_k)\ni s\mapsto&\,G_k(s)\\
	&\,=\{\delta\geqslant 1/k: \Phi_1(s,[1-\delta,1+\delta])\subset\Phi_1(s,1)+[-1,1]\}.
\end{align*}
By a standard measurable selection theorem (see, for instance, \cite{Clarke_book2013}), for each $k$, there exists a measurable selection $g_k:\text{dom}\,(G_k)\to\R$ such that $g_k(s)\in G_k(s)$ for all $s\in[0,t]\cap\text{dom}\,(G_k)$. Notice that we can assume that the sequence $\{g_k\}$ is nondecreasing and converges to a measurable selection $g$ of $G$ as $k\to\infty$. Thus, we can assume $\delta(\cdot)$ is measurable and $\delta(s)>0$ for almost all $s\in[0,t]$. We conclude that, if $\alpha\in\Omega$ satisfies $|\alpha(s)-1|\leqslant\delta(s)$ almost all $s\in[0,t]$, then
\begin{equation}\label{eq:sumability1}
	L(s,\xi(s),\dot{\xi}(s)/\alpha(s),u_{\xi}(s))\alpha(s)\in L^1([0,t]).
\end{equation}

Set
\begin{align*}
	\Omega_0=\{\alpha\in\Omega: |\alpha(s)-1|<\delta(s)\ \text{a.e $s\in[0,t]$}\}
\end{align*}
For any $M>0$ we define $E_M=\{s\in[0,t]: \dot{\xi}(s)\ \text{exists and}\ |\dot{\xi}(s)|\leqslant M\}$. We choose $M$ such that $|E_M|>0$ and $|E_M^c|>0$ and set $\beta_M=|E_M|\cdot\mathbbm{1}_{E_M^c}-|E_M^c|\cdot\mathbbm{1}_{E_M}$. Notice that $\alpha=1+\varepsilon\beta_M$ satisfies the summability condition \eqref{eq:sumability1} for small $\varepsilon>0$ even if we cannot ensure $1+\varepsilon\beta_M\in\Omega_0$.

Fix $s\in[0,t]$ such that $\dot{\xi}$ exists and $\alpha\in\Omega_0$. Given $a\in\R$ we define $f(\lambda)=L(s+\lambda a,\xi(s),\dot{\xi}(s)/\alpha(s),u_{\xi}(s))$ for $\lambda\in[0,t]$. Without loss of generality, we suppose $f(\lambda)\geqslant0$ for all $\lambda\in[0,t]$ by Proposition \ref{a_priori} and condition (L2). Invoking condition (L4) we obtain that for all $\lambda\in[0,1]$
\begin{align*}
	f'(\lambda)=&\,L_t(s+\lambda a,\xi(s),\dot{\xi}(s)/\alpha(s),u_{\xi}(s))\cdot a\\
	\leqslant&\,[C_1+C_2L(s+\lambda a,\xi(s),\dot{\xi}(s)/\alpha(s),u_{\xi}(s))]\cdot a=[C_1+C_2f(\lambda)]\cdot a\\
	\leqslant&\,C_1|a|+C_2|a|f(\lambda).
\end{align*}
Applying Gronwall's inequality we have that for all $\lambda\in[0,1]$
\begin{align*}
	f(\lambda)\leqslant e^{C_2|a|\lambda}f(1)+C_1|a|\int^{\lambda}_0e^{C_2|a|(\lambda-s)}\ ds<e^{C_2|a|}(f(1)+C_1|a|).
\end{align*}
It follows that
\begin{equation}\label{eq:sumability2}
	L(s+\varepsilon(\tau(s)-s),\xi(s),\dot{\xi}(s)/\alpha(s),u_{\xi}(s))\ \text{is bounded by a function in}\ L^1([0,t])
\end{equation}
for any $\varepsilon\in[0,1]$ provided $\alpha\in\Omega_0$ or $\alpha=1+\varepsilon\beta_M$.

\medskip

\noindent\textbf{Step III: A necessary condition.} Fix $0\not=\beta\in L^{\infty}([0,t])$ such that $1+\beta\in\Omega_0$. For any $\varepsilon\in \R$ such that $|\varepsilon|\leqslant1$ we have that $1+\varepsilon\beta\in\Omega_0\subset\Omega$. Let $\gamma(s)=\int^s_0\beta(r)\ dr$. Define the functional $\Lambda:\Omega\to\R$ by
\begin{align*}
	\Lambda(\alpha)=u_{\xi,\alpha}(t)
\end{align*}
with $u_{\xi,\alpha}$ defined in \eqref{eq:u_alpha}. Since $\Lambda(1+\varepsilon\beta)\geqslant\Lambda(1)$ for  $|\varepsilon|\leqslant1$, we have that $\frac d{d\varepsilon}\Lambda(1+\varepsilon\beta)\vert_{\varepsilon=0}=0$ if the derivative exists. Thus, for $\varepsilon>0$, by \eqref{eq_implicit},
\begin{equation}\label{eq:upper1}
 	0\leqslant\frac {\Lambda(1+\varepsilon\beta)-\Lambda(1)}{\varepsilon}=\int^t_0e^{\int^t_s\widehat{L_u^{\varepsilon}}dr}\lambda_{\varepsilon}(s)\ ds,
\end{equation}
where $\widehat{L_u^{\varepsilon}}=\widehat{L_u^{1+\varepsilon\beta}}$ and
\begin{equation}\label{eq:spliting}
	\lambda_{\varepsilon}(s):=\frac{L(s+\varepsilon\gamma,\xi,\dot{\xi}/(1+\varepsilon\beta),u_{\xi})(1+\varepsilon\beta)-L(s,\xi,\dot{\xi},u_{\xi})}{\varepsilon}.
\end{equation}
Set
\begin{align*}
	l_{\varepsilon}(s):=L_v(s,\xi,\dot{\xi}/(1+\varepsilon\beta),u_{\xi})\cdot\dot{\xi}/(1+\varepsilon\beta)-L(s,\xi,\dot{\xi}/(1+\varepsilon\beta),u_{\xi}).
\end{align*}
For convenience we take out the variable $s$ on right side of the inequalities above. We claim that
\begin{equation}\label{eq:dBR}
    0=\frac d{d\varepsilon}\Lambda(1+\varepsilon\beta)\vert_{\varepsilon=0}=\int^t_0e^{\int^t_sL_udr}\,\left\{l_0\cdot\beta-L_t(s,\xi,\dot{\xi},u_{\xi})\cdot\gamma\right\}\ ds.
\end{equation}

\medskip

\noindent\textbf{Step IV: On the summability.} By convexity we have that
\begin{align*}
	&\,L(s,\xi,\dot{\xi}/(1+\varepsilon\beta),u_{\xi})-L(s,\xi,\dot{\xi},u_{\xi})\\
	\leqslant&\,-L_v(s,\xi,\dot{\xi}/(1+\varepsilon\beta),u_{\xi})\cdot\{\dot{\xi}-\dot{\xi}/(1+\varepsilon\beta)\}\\
	=&\,-\varepsilon\beta L_v(s,\xi,\dot{\xi}/(1+\varepsilon\beta),u_{\xi})\cdot\dot{\xi}/(1+\varepsilon\beta).
\end{align*}
It follows that
\begin{equation}\label{eq:upper2}
	\begin{split}
		\lambda_{\varepsilon}\leqslant&\,-\beta\{L_v(s,\xi,\dot{\xi}/(1+\varepsilon\beta),u_{\xi})\cdot\dot{\xi}/(1+\varepsilon\beta)-L(s,\xi,\dot{\xi}/(1+\varepsilon\beta),u_{\xi})\}\\
		&\,+\beta\{L(s+\varepsilon\gamma,\xi,\dot{\xi}/(1+\varepsilon\beta),u_{\xi})-L(s,\xi,\dot{\xi}/(1+\varepsilon\beta),u_{\xi})\}\\
		&\,+\frac 1{\varepsilon}(L(s+\varepsilon\gamma,\xi,\dot{\xi}/(1+\varepsilon\beta),u_{\xi})-L(s,\xi,\dot{\xi}/(1+\varepsilon\beta),u_{\xi}))\\
	=&\,-\beta\cdot l_{\varepsilon}+\beta\cdot b_{\varepsilon}+\frac{b_{\varepsilon}}{\varepsilon},
	\end{split}
\end{equation}
where
\begin{align*}
	b_{\varepsilon}=L(s+\varepsilon\gamma,\xi,\dot{\xi}/(1+\varepsilon\beta),u_{\xi})-L(s,\xi,\dot{\xi}/(1+\varepsilon\beta),u_{\xi}).
\end{align*}

In order to validate the integrand on the right side of \eqref{eq:upper1}, we need to focus on the summability of $\lambda_{\varepsilon}$. We observe that
\begin{align*}
	L(s+\varepsilon\gamma,\xi,\dot{\xi}/(1+\varepsilon\beta),u_{\xi})-L(s,\xi,\dot{\xi}/(1+\varepsilon\beta),u_{\xi})=\int^1_0\widehat{L_t}(\lambda,s)\ d\lambda\cdot\varepsilon \gamma,
\end{align*}
where for $\lambda\in[0,1]$ we denote
\begin{align*}
	\widehat{L_t}(\lambda,s):=L_t(s+\lambda\varepsilon\gamma,\xi,\dot{\xi}/(1+\varepsilon\beta),u_{\xi})
\end{align*}
Due to condition (L4), we have that
\begin{align*}
	&\,\frac 1{\varepsilon}\left|L(\tau,\xi,\dot{\xi}/(1+\varepsilon\beta),u_{\xi})-L(s,\xi,\dot{\xi}/(1+\varepsilon\beta),u_{\xi})\right|\\
	\leqslant&\,|\gamma|\cdot\int^1_0C_1+C_2L(s+\lambda\varepsilon\gamma,\xi,\dot{\xi}/(1+\varepsilon\beta),u_{\xi})\ d\lambda.
\end{align*}
Thus $b_{\varepsilon}(s)/\varepsilon$ is bounded by an $L^1$-function by \eqref{eq:sumability2}.

In view of Proposition \ref{convexity} (a), we have that
\begin{align*}
	l_{\varepsilon}\geqslant -L(s,\xi,0,u_{\xi})\geqslant-\overline{\theta}_0(0)-c_1-KF(t,|y-x|/t).
\end{align*}
For any $\beta\in\Omega_0$ and $\varepsilon\in[0,1]$, we rewrite $\lambda_{\varepsilon}(s)$, $l_{\varepsilon}(s)$ and $b_{\varepsilon}(s)$ as $\lambda_{\varepsilon}^{\beta}(s)$, $l_{\varepsilon}^{\beta}(s)$ and $b^{\beta}_{\varepsilon}(s)$ respectively.

Set $\beta^+=\beta\cdot\mathbbm{1}_{\{\beta\geqslant0\}}$ and $\beta^-=-\beta\cdot\mathbbm{1}_{\{\beta<0\}}$, then
\begin{align*}
	\beta=\beta^+-\beta^-,\quad\text{and}\quad\beta^{\pm}\geqslant0.
\end{align*}
By \eqref{eq:upper2} we have that
\begin{align*}
	\lambda_{\varepsilon}^{\beta}(s)+\beta^+(s)l_{\varepsilon}^{\beta}(s)-\beta(s)b^{\beta}_{\varepsilon}(s)-\frac{b^{\beta}_{\varepsilon}(s)}{\varepsilon}\leqslant\beta^-(s)l_{\varepsilon}^{\beta}(s).
\end{align*}
Now, observe that $\beta^+(s)l_{\varepsilon}^{\beta}(s)=\beta^+(s)l_{\varepsilon}^{\beta^+}(s)$ and $\beta^-(s)l_{\varepsilon}^{\beta}(s)=\beta^-(s)l_{-\varepsilon}^{\beta^-}(s)$. Then the inequalities above can recast as follows
\begin{equation}\label{eq:integrability2}
	\lambda_{\varepsilon}^{\beta}(s)+\beta^+(s)l_{\varepsilon}^{\beta^+}(s)-\beta(s)b^{\beta}_{\varepsilon}(s)-\frac{b^{\beta}_{\varepsilon}(s)}{\varepsilon}\leqslant\beta^-(s)l_{-\varepsilon}^{\beta^-}(s).
\end{equation}
Lemma \ref{convexity} (a) ensures that $\varepsilon\mapsto l_{\varepsilon}^{\beta^-}$ is decreasing on $[-1,1]$ and we conclude that
\begin{equation}
	\beta^-l_{-\varepsilon}^{\beta^-}\leqslant\beta^-l_{-1}^{\beta^-}\quad\forall\varepsilon\in(0,1).
\end{equation}
By Lemma \ref{convexity} (b), we obtain
\begin{align*}
	\beta^-l_{-\varepsilon}^{\beta^-}=&\,\beta^-\{L_v(s,\xi,\dot{\xi}/(1-\varepsilon\beta^-),u_{\xi})\cdot\dot{\xi}/(1-\varepsilon\beta^-)-L(s,\xi,\dot{\xi}/(1-\varepsilon\beta^-),u_{\xi})\}\\
	\leqslant&\,(\kappa_{\varepsilon}^{\beta^-})^{-1}L(s,\xi,\dot{\xi}/(1-\beta^-),u_{\xi})-((\kappa_{\varepsilon}^{\beta^-})^{-1}+\beta^-)L(s,\xi,\dot{\xi}/(1-\varepsilon\beta^-),u_{\xi})
\end{align*}
where $(\kappa_{\varepsilon}^{\beta^-})^{-1}=\frac{1-\beta^-}{1-\varepsilon}$. In view of \eqref{eq:sumability2}, \eqref{eq:integrability2} and the fact that $(\kappa_{\varepsilon}^{\beta^-})^{-1}$ is bounded, we conclude that $\beta^-l_{-\varepsilon}^{\beta^-}\in L^1([0,t])$ for all $\varepsilon\in(0,1]$ uniformly.

\medskip

\noindent\textbf{Step IV: Erdmann condition.} We rewrite $L_t(s)=L_t(s,\xi(s),\dot{\xi}(s),u_{\xi}(s))$. Recalling that for almost all $s\in[0,t]$ we have that
\begin{align*}
	\lim_{\varepsilon\to0^+}b^{\beta}_{\varepsilon}(s)=0,\quad\lim_{\varepsilon\to0^+}\frac{b^{\beta}_{\varepsilon}(s)}{\varepsilon}=L_t(s)\cdot\gamma(s).
\end{align*}
Thus, integrating \eqref{eq:integrability2}, by Lebesgue's theorem we obtain
\begin{align*}
	\int^t_0e^{\int^t_sL_udr}\{l_0(s)\beta^+(s)-L_t(s)\gamma(s)\}\ ds\leqslant\int^t_0e^{\int^t_sL_udr}l_0(s)\beta^-(s)ds.
\end{align*}
Therefore, $\int^t_0e^{\int^t_sL_udr}\{l_0\cdot\beta-L_t\cdot\gamma\}\ ds\leqslant0$ and \eqref{eq:dBR} follows since $\beta\in\Omega_0$ is arbitrary. 

Now, observe that  the primitive $\mu(s):=\int^s_0\beta(r)dr$ gives a one-to-one correspondence between  $\Omega_0$
and the set
\begin{align*}
	\Omega_1=\{\mu:[0,t]\to\R: \mu\ \text{is Lipschitz continuous with}\ \mu(0)=\mu(t)=0, \mu'\in\Omega_0\}.
\end{align*}
Thus, \eqref{eq:dBR} can be recast as follows
\begin{align*}
	0=-e^{\int^t_0L_udr}\,\int^t_0E(s)\mu'(s)-e^{-\int^s_0L_udr}L_t(s)\mu(s)\ ds\quad\forall \mu\in\Omega_1.
\end{align*}
So, \eqref{eq:Erdmann_condition} follows by the generalized du Bois-Reymond lemma \footnote{By \eqref{eq:sumability2}, the previous steps of the proof can also be applied to $\beta_M$. This shows that \eqref{eq:dBR} holds for $\beta_M$ and this leads to the summability of $l_0$ as well as $E$. This allows us to use Theorem \ref{du_Bois_Reymond1}.} (Theorem \ref{du_Bois_Reymond1}).
\end{proof}


\subsection{Herglotz equation}

\begin{The}[Herglotz equation]\label{thm:EL}
	Suppose conditions \mbox{\rm (L1)-(L3)} and \mbox{\rm (L4')} are satisfied. Let $\xi\in\Gamma^{a,b}_{x,y}$ be a minimizer of \eqref{eq:app_fundamental_solution} with $u_{\xi}$ determined by \eqref{eq:app_caratheodory_L}. Then, the function $s\mapsto L_x(s,\xi(s),\dot{\xi}(s),u_{\xi}(s))$ is absolutely continuous on $[a,b]$ and $(\xi,u_{\xi})$ satisfies the Carath\'eodory equation \eqref{eq:app_caratheodory_L} and the Herglotz equation
\begin{equation}\label{eq:Herglotz}
	\begin{split}
		&\,\frac d{ds}L_v(s,\xi(s),\dot{\xi}(s),u_{\xi}(s))\\
	=&\,L_x(s,\xi(s),\dot{\xi}(s),u_{\xi}(s))+L_u(s,\xi(s),\dot{\xi}(s),u_{\xi}(s))L_v(s,\xi(s),\dot{\xi}(s),u_{\xi}(s))
	\end{split}
\end{equation}
for almost all $s\in[a,b]$.
\end{The}

\begin{proof}
Let $\xi\in\Gamma^{a,b}_{x,y}$ be a minimizer of \eqref{eq:app_fundamental_solution} where $u_{\xi}$ is determined uniquely by \eqref{eq:app_caratheodory_L}. For any $\varepsilon\in\R$ and any Lipschitz function $\eta\in\Gamma^{a,b}_{0,0}$,  we set $\xi_{\varepsilon}(s)=\xi(s)+\varepsilon\eta(s)$. Let $u_{\xi_{\varepsilon}}$ be the associated unique solution of \eqref{eq:app_caratheodory_L} with respect to $\xi_{\varepsilon}$, i.e., $u_{\xi_{\varepsilon}}$ satisfies
\begin{equation}\label{eq:Caratheory_lambda}
	\begin{cases}
		\dot{u}_{\xi_{\varepsilon}}(s)=L(s,\xi_{\varepsilon}(s),\dot{\xi}_{\varepsilon}(s),u_{\xi_{\varepsilon}}(s)),\quad a.e.\ s\in[a,b],&\\
		u_{\xi_{\varepsilon}}(0)=u.&
	\end{cases}
\end{equation}
It is clear that $\xi_{\varepsilon}\in\Gamma^{a,b}_{x,y}$ and $J(\xi)\leqslant J(\xi_{\varepsilon})$. Combining \eqref{eq:Caratheory_lambda} and \eqref{eq:app_caratheodory_L} we have that
\begin{align*}
	\dot{u}_{\xi_{\varepsilon}}-\dot{u_{\xi}}=&\,L(s,\xi_{\varepsilon},\dot{\xi}_{\varepsilon},u_{\xi_{\varepsilon}})-L(s,\xi,\dot{\xi},u_{\xi})\\
	=&\,\{L(s,\xi_{\varepsilon},\dot{\xi}_{\varepsilon},u_{\xi_{\varepsilon}})-L(s,\xi_{\varepsilon},\dot{\xi}_{\varepsilon},u_{\xi})\}+\{L(s,\xi_{\varepsilon},\dot{\xi}_{\varepsilon},u_{\xi})-L(s,\xi,\dot{\xi},u_{\xi})\}\\
	=&\,\widehat{L^{\varepsilon}_u}(u_{\xi_{\varepsilon}}-u_{\xi})+\{L(s,\xi_{\varepsilon},\dot{\xi}_{\varepsilon},u_{\xi})-L(s,\xi,\dot{\xi},u_{\xi})\},
\end{align*}
where
\begin{align*}
	\widehat{L^{\varepsilon}_u}(s)=\int^1_0L_u(s,\xi_{\varepsilon}(s),\dot{\xi}_{\varepsilon}(s),u_{\xi}(s)+\lambda(u_{\xi_{\varepsilon}}(s)-u_{\xi}(s)))\ d\lambda.
\end{align*}
It follows that
\begin{align*}
	u_{\xi_{\varepsilon}}(s)-u_{\xi}(s)=\int^s_ae^{\int^s_{\sigma}\widehat{L_u^{\varepsilon}}dr}(L(s,\xi_{\varepsilon},\dot{\xi}_{\varepsilon},u_{\xi})-L(s,\xi,\dot{\xi},u_{\xi}))\ d\sigma.
\end{align*}
Recalling that $J(\xi_{\varepsilon})=u_{\xi}(t)$, we obtain
\begin{equation}\label{eq:variation1}
	0\leqslant \frac {J(\xi_{\varepsilon})-J(\xi)}{\varepsilon}=\int^b_ae^{\int^b_s\widehat{L_u^{\varepsilon}}dr}\cdot\frac{L(s,\xi_{\varepsilon},\dot{\xi}_{\varepsilon},u_{\xi})-L(s,\xi,\dot{\xi},u_{\xi})}{\varepsilon}\ ds.
\end{equation}

Now, similarly to Step II of the proof of Theorem \ref{Erdmann_condition}, by using the measurable selection theorem, there exists $\delta\in L^{\infty}([a,b])$, with $\delta>0$ a.e., such that, if $|\eta(s)|\leqslant\delta(s)$ for almost all $s\in[a,b]$, then $L(s,\xi_{\varepsilon},\dot{\xi}_{\varepsilon},u_{\xi})$ is bounded by an $L^1$-function uniformly for $|\varepsilon|\leqslant1$. Invoking condition (L4'), we conclude that $L_x(s,\xi_{\varepsilon},\dot{\xi},u_{\xi})$ is also bounded by an $L^1$-function uniformly for $|\varepsilon|\leqslant1$. By convexity we have that
\begin{align*}
	L(s,\xi,\dot{\xi},u_{\xi})-L(s,\xi,\dot{\xi}_{-1},u_{\xi})\leqslant L_v(s,\xi,\dot{\xi},u_{\xi})\cdot\dot{\eta}\leqslant L(s,\xi,\dot{\xi}_1,u_{\xi})-L(s,\xi,\dot{\xi},u_{\xi}).
\end{align*}
It follows that $L_v(s,\xi,\dot{\xi},u_{\xi})\cdot\dot{\eta}\in L^1([a,b])$. Now, we can assume that  $L_v(s,\xi,\dot{\xi}_{\varepsilon},u_{\xi})\cdot\dot{\eta}$ is bounded by an $L^1$-function for all $|\varepsilon|\leqslant1$.

Fix $\eta\in\Gamma^{a,b}_{0,0}$ such that $|\eta(s)|\leqslant\delta(s)$ for almost all $s\in[0,t]$. We claim that
\begin{equation}\label{eq:EL_integral}
	\frac{d}{d\varepsilon}J(\xi_{\varepsilon})=0=\int^b_ae^{\int^b_sL_udr}\cdot\{L_x\cdot\eta+L_v\cdot\dot{\eta}\}\ ds.
\end{equation}
By convexity, we have that
\begin{align*}
	L_v(s,\xi,\dot{\xi},u_{\xi})\cdot\dot{\eta}\leqslant\frac{L(s,\xi,\dot{\xi}_{\varepsilon},u_{\xi})-L(s,\xi,\dot{\xi},u_{\xi})}{\varepsilon}\leqslant L_v(s,\xi,\dot{\xi}_{\varepsilon},u_{\xi})\cdot\dot{\eta}.
\end{align*}
Moreover,
\begin{align*}
	\left|\frac{L(s,\xi_{\varepsilon},\dot{\xi}_{\varepsilon},u_{\xi})-L(s,\xi,\dot{\xi}_{\varepsilon},u_{\xi})}{\varepsilon}\right|\leqslant |\eta|\int^1_0|L_x(s,\xi_{\varepsilon}+\lambda(\xi_{\varepsilon}-\xi),\dot{\xi}_{\varepsilon},u_{\xi})|\ d\lambda
\end{align*}
Taking the limit in \eqref{eq:variation1} as $\varepsilon\to 0^+$, then \eqref{eq:EL_integral} follows by Lebesgue's theorem. Thus, \eqref{eq:Herglotz} follows by Theorem \ref{du_Bois_Reymond1} provided $L_v(s,\xi(s),\dot{\xi}(s),u_{\xi}(s))\in L^1([a,b])$ which is guaranteed by condition (L4').
\end{proof}

\begin{Rem}
It is also useful to rewrite the Herglotz equation is the form
\begin{equation}\label{eq:Herglotz_2}
	\frac d{ds}e^{-\int^s_aL_u(r)\ dr}L_v(s,\xi(s),\dot{\xi}(s),u_{\xi}(s))=e^{-\int^s_aL_u(r)\ dr}L_x(\xi(s),\dot{\xi}(s),u_{\xi}(s)),
\end{equation}
where $L_u(s)=L_u(s,\xi(s),\dot{\xi}(s),u_{\xi}(s))$.
\end{Rem}


\subsection{Lipschitz estimates}

In this section, we will prove the Lipschitz estimates for the minimizer $\xi$ of \eqref{eq:app_fundamental_solution}.

\begin{The}\label{Lip}
Suppose conditions \mbox{\rm (L1)-(L3)} are satisfied together with either \mbox{\rm (L4)} or \mbox{\rm (L4')}. Let $u\in\R$  and $R>0$ be fixed. Then there exists  a continuous function $F=F_{u,R}:[0,+\infty)\times[0,+\infty)\to[0,+\infty)$, with $F(t,r)$ nondecreasing in both variables and superlinear with respect to $r$, such that for any given $b>a$ and $x,y\in\R^n$, with $|x-y|\leqslant R$, every minimizer $\xi\in\Gamma^{a,b}_{x,y}$ for \eqref{eq:app_fundamental_solution} satisfies
	\begin{align*}
		\operatorname*{ess\ sup}_{s\in[a,b]}|\dot{\xi}(s)|\leqslant F(b-a,R/(b-a)).
	\end{align*}
\end{The}

\begin{proof}
We consider two cases, one for each of the different assumptions of the theorem.

\medskip

\noindent\textbf{Case I}: We assume conditions (L1)-(L3) together with (L4).

\medskip

Let $\xi\in\Gamma^{a,b}_{x,y}$ be a minimizer of \eqref{eq:app_fundamental_solution}, for $\alpha>0$. Set
\begin{align*}
	l_{\xi}(s,\alpha)=\alpha\cdot L(s,\xi(s),\dot{\xi}(s)/\alpha,u_{\xi}(s))
\end{align*}
and recall $E_0=L_v(s,\xi,\dot{\xi},u_{\xi})\cdot\dot{\xi}-L(s,\xi,\dot{\xi},u_{\xi})$. Simple computations show that $l_{\xi}(s,\cdot)$ is convex and
\begin{align*}
    \frac{d}{d\alpha}\bigg|_{\alpha=1}l_{\xi}(s,\alpha)=-E_{0}(s).
\end{align*}
Choosing $s_{0}\in[a,b]$ such that $|\dot{\xi}(s_{0})|=\operatorname*{ess\ inf}_{s\in[a,b]}|\dot{\xi}(s)|$ by convexity, we have that
$$
-E_{0}(s_{0})\geqslant \sup_{\alpha<1}\frac{l_{\xi}(s_{0},1)-l_{\xi}(s_{0},\alpha)}{1-\alpha}
$$
Recall that $|u_\xi|$ is bounded by $F_1(b-a,R/(b-a))$ and $\operatorname*{ess\ inf}_{s\in[a,b]}|\dot{\xi}(s)|$ is bounded by $F_2(b-a,R/(b-a))$ by Proposition \ref{a_priori}. For convenience, we drop the variables in the functions $F_1$ and $F_2$, and also $F_i$ in the following text.

Taking $\alpha=\frac{1}{2}$, by (L2)-(L3) we conclude that
\begin{align*}
	-E_{0}(s_0)\geqslant&\,2(l_{\xi}(s_0,1)-l_{\xi}(s_0,1/2))=2(L(s_{0},\xi(s_0),\dot{\xi}(s_0),u_{\xi}(s_0))-l_{\xi}(s_0,1/2))\\
		\geqslant&\,-2c_0-2KF_1-L(s_{0},\xi(s_0),2\dot{\xi}(s_0),u_{\xi}(s_0)) \\
		\geqslant&\,-2c_0-3KF_1-L(s_{0},\xi(s_0),2\dot{\xi}(s_0),0)\\
		\geqslant&\,-2c_0-3KF_1-\overline{\theta}_0(2|\dot{\xi}(s_0)|)-c_1\\
		\geqslant&\,-2c_0-3KF_1-\overline{\theta}_0(2F_2)-c_1:=-F_3.
\end{align*}
We rewrite $L_t(s)=L_t(s,\xi(s),\dot{\xi}(s),u_{\xi}(s))$ and $L_u(s)=L_u(s,\xi(s),\dot{\xi}(s),u_{\xi}(s))$. Then, by Erdmann's condition \eqref{eq:Erdmann_condition} we obtain that for almost all $s\in[a,b]$,
\begin{align*}
	E(s)=&\,E(s_0)-\int^{s}_{s_{0}}e^{-\int^{\tau}_{a}L_udr}L_t(\tau)\ d\tau\leqslant e^{-\int^{s_0}_{a}L_udr}E_0(s_0)+\int^b_ae^{-\int^{\tau}_{a}L_udr}|L_t(\tau)|\ d\tau\\
	\leqslant&\,e^{K(b-a)}F_3+e^{K(b-a)}\int^b_a|L_t(s)|\ ds.
\end{align*}
By (L4) we conclude that
\begin{align*}
	E(s)\leqslant&\,e^{K(b-a)}F_3+e^{K(b-a)}\int^b_a\big\{C_1+C_2L(s,\xi(s),\dot{\xi}(s),u_{\xi}(s))\big\}\ ds\\
	\leqslant&\,e^{K(b-a)}\big\{F_3+C_1(b-a)+C_2F_4\big\}:=F_5,
\end{align*}
where $\int^b_aL(s,\xi(s),\dot{\xi}(s),u_{\xi}(s))\ ds$ is bounded by $F_4$ by Proposition \ref{a_priori}. Therefore, we have that, for almost all $s\in[a,b]$,
\begin{equation}\label{eq:Erdmann_condition1}
E_{0}(s)=e^{\int^s_aL_ud\tau}E(s)\leqslant e^{K(b-a)}F_5:=F_6.
\end{equation}

Now, let $s$ be such that $\dot{\xi}(s)$ exists and \eqref{eq:Erdmann_condition1} holds. By convexity, we have that
\begin{align*}
	&\,L(s,\xi(s),\dot{\xi}(s)/(1+|\dot{\xi}(s)|),u_{\xi}(s))-L(s,\xi(s),\dot{\xi}(s),u_{\xi}(s))\\
	\geqslant&\,((1+|\dot{\xi}(s)|)^{-1}-1)\cdot\langle L_v(s,\xi(s),\dot{\xi}(s),u_{\xi}(s)),\dot{\xi}(s)\rangle\\
	\geqslant&\,((1+|\dot{\xi}(s)|)^{-1}-1)\cdot(L(s,\xi(s),\dot{\xi}(s),u_{\xi}(s))+F_6).
\end{align*}
It follows that
\begin{align*}
	&\,L(s,\xi(s),\dot{\xi}(s),u_{\xi}(s))\\
	\leqslant&\, L(s,\xi(s),\dot{\xi}(s)/(1+|\dot{\xi}(s)|),u_{\xi}(s))(1+|\dot{\xi}(s)|)+F_6|\dot{\xi}(s)|.
\end{align*}
Let $C=\sup_{s\in[a,b],|v|\leqslant1}L(s,\xi(s),v,u_{\xi}(s))$ and observe that, by (L2) and Proposition \ref{a_priori},
\begin{align*}
	C\leqslant\sup_{s\in[a,b],|v|\leqslant1}\{L(s,\xi(s),v,0)+K|u_{\xi}(s)|\}\leqslant\overline{\theta}_0(1)+c_1+KF_1:=F_7.
\end{align*}
It follows that
\begin{align*}
	L(s,\xi(s),\dot{\xi}(s),u_{\xi}(s))\leqslant F_7+(F_6+F_7)|\dot{\xi}(s)|.
\end{align*}
Therefore, invoking Proposition \ref{a_priori}, we obtain
\begin{align*}
	&\,(F_6+F_7+1)|\dot{\xi}(s)|-(\theta_0^*(F_6+F_7+1)+c_0)\\
	\leqslant&\, \theta_0(|\dot{\xi}(s)|)-c_0\leqslant L(s,\xi(s),\dot{\xi}(s),0)\leqslant L(s,\xi(s),\dot{\xi}(s),u_{\xi}(s))+K|u_{\xi}(s)|\\
	\leqslant&\, F_7+(F_6+F_7)|\dot{\xi}(s)|+KF_1.
\end{align*}
This leads to
\begin{align*}
	|\dot{\xi}(s)|\leqslant(\theta_0^*(F_6+F_7+1)+c_0)+F_7+KF_1:=F_8,
\end{align*}
which completes the proof of Case I.

\medskip

\noindent\textbf{Case II}: We suppose conditions (L1)-(L3) together with (L4') are satisfied.

\medskip

This case is much easier than Case I. Again, we choose $s_{0}\in[0,t]$ such that
$$
|\dot{\xi}(s_{0})|=\operatorname*{ess\ inf}_{s\in[a,b]}|\dot{\xi}(s)|\leqslant F_1.
$$
By Corollary \ref{a_priori}, $\xi(s)$ is contained in $B(x,(b-a)F_2)$ and $|u_{\xi}(s)|$ is bounded by $F_3$. Set
\begin{align*}
	F_4=\max\{|L_{v}(s_{0},y,v,r)|:|y-x|\leqslant (b-a)F_2,|v|\leqslant F_1,|r|\leqslant F_3\}.
\end{align*}
By solving Herglotz' equation in the form \eqref{eq:Herglotz_2} we have that, for any $s\in[a,b]$,
\begin{align*}
	&\,e^{-\int^{s}_aL_ud\tau}L_{v}(s,\xi(s),\dot{\xi}(s),u_{\xi}(s))\\
	=&\,e^{-\int^{s_{0}}_aL_ud\tau}L_{v}(s_{0},\xi(s_{0}),\dot{\xi}(s_{0}),u_{\xi}(s_{0}))+\int^{s}_{s_{0}}e^{-\int^{\tau}_aL_u dr}L_{x}\ d\tau
\end{align*}
By condition (L4') we conclude that, for almost $s\in[a,b]$,
\begin{equation}\label{eq:Erdmann_condition2}
	\begin{split}
		&\,|L_{v}(s,\xi(s),\dot{\xi}(s),u_{\xi}(s))|\\
	\leqslant&\,e^{2K(b-a)}F_4+e^{2K(b-a)}\int^b_a|L_x(s,\xi(s),\dot{\xi}(s),u_{\xi}(s))|\ ds\\
	\leqslant&\,e^{2K(b-a)}F_4+e^{2K(b-a)}\int^b_a\big\{C_1+C_2L(s,\xi(s),\dot{\xi}(s),u_{\xi}(s))\big\}\ ds\\
	\leqslant&\,e^{2K(b-a)}\big\{F_4+C_1(b-a)+C_2F_5\big\}:=F_6.
	\end{split}
\end{equation}

Now, let $H$ be the Hamiltonian associated with $L$. Set
\begin{align*}
	F_7=\max\{|H_{p}(s,y,p,r)|:|y-x|\leqslant (b-a)F_2,|p|\leqslant F_6,|r|\leqslant F_3\}.
\end{align*}
Then, for any $s\in[a,b]$ such that $\dot{\xi}(s)$ exists and \eqref{eq:Erdmann_condition2} is satisfied, we obtain that
\begin{align*}
	|\dot{\xi}(s)|=|H_p(s,\xi(s),L_v(s,\xi(s),\dot{\xi}(s),u_{\xi}(s)),u_{\xi}(s))|\leqslant F_7.
\end{align*}
This completes the proof of Case II.
\end{proof}

\begin{Cor}
Theorem \ref{thm:EL} holds under the assumptions \mbox{\rm (L1)-(L4)} or \mbox{\rm (L1)-(L3) and (L4')}. In particular, Herglotz equation \eqref{eq:Herglotz} holds true.
\end{Cor}

\begin{proof}
Due to Theorem \ref{Lip}, we have the uniform bound of $\dot{\xi}(s)$ for almost all $s\in[0,t]$. Along the proof of Theorem \ref{Erdmann_condition}, there is no summability difficulty since the Lipschitz estimates, and Erdmann condition \eqref{eq:Erdmann_condition} can be obtained directly by Step V in the proof of Theorem \ref{Erdmann_condition}. Now, the proof of the theorem is similar to but simpler than that of Theorem \ref{thm:EL}  because of our Lipschitz estimates.
\end{proof}

\begin{Cor}
	The minimal curve $\xi$ of \eqref{eq:app_fundamental_solution} is of class $C^1$ as well as $u_{\xi}$.
\end{Cor}

\begin{proof}
Let $N$ be the set of zero Lebesgue measure where $\dot{\xi}$ does not exist.  For $\bar{t}\in[a,b]$, choose a sequence $\{t_k\}\in [a,b]\setminus N$ such that $t_k\to\bar{t}$. Then $\dot{\xi}(t_k)\to\bar{v}$ for some $\bar{v}\in\R^n$ (up to subsequences) and
	\begin{align*}
		&\,L_v(\bar{t},\xi(\bar{t}),\dot{\xi}(\bar{t}),u_{\xi}(\bar{t}))-L_v(t_1,\xi(t_1),\dot{\xi}(t_1),u_{\xi}(t_1))\\
		=&\,\lim_{k\to\infty}L_v(t_k,\xi(t_k),\dot{\xi}(t_k),u_{\xi}(t_k))-L_v(t_1,\xi(t_1),\dot{\xi}(t_1),u_{\xi}(t_1))\\
		=&\,\int^{\bar{t}}_{t_1}\{L_x(s,\xi(s),\dot{\xi}(s),u_{\xi}(s))+L_u(s,\xi(s),\dot{\xi}(s),u_{\xi}(s))L_v(s,\xi(s),\dot{\xi}(s),u_{\xi}(s))\} ds
	\end{align*}
	by Herglotz equation \eqref{eq:Herglotz}. From the strict convexity of $L$ it follows that the map $v\mapsto L_v(s,\xi(s),v,u_{\xi}(s))$ is a diffeomorphism. This implies that $\bar{v}$ is uniquely determined, i.e.,
	\begin{align*}
		\lim_{[0,t]\setminus N\ni s\to\bar{t}}\dot{\xi}(s)=\bar{v}.
	\end{align*}
	Now, by Lemma 6.2.6 in \cite{Cannarsa_Sinestrari_book}, $\dot{\xi}(\bar{t})$ exists and $\lim_{[0,t]\setminus N\ni s\to\bar{t}}\dot{\xi}(s)=\dot{\xi}(\bar{t})$. It follows that $\xi$ is of class $C^1$. In view of \eqref{eq:app_caratheodory_L}, $u_{\xi}$ is also of class $C^1$.
\end{proof}

The following improvement of the main results in this section is very similar to that in \cite{CCWY2018}. We omit the proof.

\begin{Pro}\label{Herglotz_Lie}
	Suppose $L$ is of class $C^2$ and satisfies conditions {\rm (L1)-(L3)} together with {\rm (L4)} or {\rm (L4')}. For any fixed $x,y\in\R^n$, $b>a$ and $u\in\R$, the functional $J$ defined in  \eqref{eq:app_fundamental_solution} admits a minimizer. Moreover,
\begin{enumerate}[\rm (a)]
  \item both $\xi$ and $u_{\xi}$ are of class $C^2$ and $\xi$ satisfies Herglotz' equation \eqref{eq:Hergotz_intro} for all $s\in[a,b]$ where $u_{\xi}$ is the unique solution of \eqref{eq:app_caratheodory_L};
  \item the dual arc $p$ defined by $p(s)=L_v(s,\xi(s),\dot{\xi}(s),u_{\xi}(s))$ is also of class $C^2$ and $(\xi,p,u_{\xi})$ satisfies Lie equation \eqref{eq:Lie} for all $s\in[a,b]$.
\end{enumerate}
\end{Pro}

\section{Applications to Hamilton-Jacobi equations in the contact type}

In this section, we want to explain the relations between Herglotz' variational principle and the Hamilton-Jacobi equation \eqref{eq:HJe_t}. Throughout this section, we suppose that $L$ satisfies condition (L1)-(L3), together with (L4) or (L4'). Therefore Proposition \ref{Herglotz_Lie} holds.

\subsection{Fundamental solutions and Lax-Oleinik evolution}

Fix $x,y\in\R^n$, $t_2>t_1$ and $u\in\R$. Let $\xi\in\mathcal{A}_1:=\Gamma^{t_1,t_2}_{x,y}\cap C^2([t_1,t_2],\R^n)$ and let $u_{\xi}$ be the unique $C^2$ solution of the ODE
\begin{equation}\label{eq:Caratheodory_ODE}
	\begin{cases}
		\dot{u}_{\xi}(s)=L(s,\xi(s),\dot{\xi}(s),u_{\xi}(s)),\quad s\in[t_1,t_2],&\\
		u_{\xi}(t_1)=u.&
	\end{cases}
\end{equation}
We define
\begin{equation}\label{eq:fundamental_solution}
	h_L(t_1,t_2,x,y,u):=\inf_{\xi\in\mathcal{A}_1}\int^{t_2}_{t_1}L(s,\xi(s),\dot{\xi}(s),u_{\xi}(s))\ ds=\inf_{\xi\in\mathcal{A}_1}u_{\xi}(t_2)-u.
\end{equation}

An associated variational problem of Herglotz' type is as follows:
\begin{equation}\label{eq:fundamental_solution_2}
	\breve{h}_L(t_1,t_2,x,y,u):=\inf_{\xi}\int^{t_2}_{t_1}L(s,\xi(s),\dot{\xi}(s),w_{\xi}(s))\ ds
\end{equation}
where the infimum is taken over all $\xi\in\mathcal{A}_1$ such that a terminal condition problem of Carath\'eodory equation
\begin{equation}\label{eq:caratheodory_L_2}
	\begin{cases}
		\dot{w}_{\xi}(s)=L(s,\xi(s),\dot{\xi}(s),w_{\xi}(s)),\quad s\in[t_1,t_2],&\\
		w_{\xi}(t_2)=u,&
	\end{cases}
\end{equation}
admits a (unique) solution. Invoking Proposition \ref{Herglotz_Lie}, the infimum in the definition of $h_L(t_1,t_2,x,y,u)$ and $\breve{h}_L(t_1,t_2,x,y,u)$ can be achieved.

\begin{defn}
	Fix $x,y\in\R^n$, $t_2>t_1$ and $u\in\R$. We call the function $h_L(t_1,t_2,x,y,u)$ (resp. $\breve{h}_L(t_1,t_2,x,y,u)$) the {\em negative} (resp. {\em positive}) {\em type fundamental solution} for \eqref{eq:HJe_t}.
\end{defn}

\begin{defn}[$t$-dependent case]\label{defn:Lax_Oleinik}
	For any function $\phi:\R^n\to[-\infty,+\infty]$, we define
	\begin{align*}
	\begin{split}
		(\mathbf{T}^{t_2}_{t_1}\phi)(x)=\inf_{y\in\R^n}\{\phi(y)+h_L(t_1,t_2,y,x,\phi(y))\},\\
		(\breve{\mathbf{T}}^{t_2}_{t_1}\phi)(x)=\sup_{y\in\R^n}\{\phi(y)-\breve{h}_L(t_1,t_2,x,y,\phi(y))\},
	\end{split}
		\quad t_2>t_1, x\in\R^n.
	\end{align*}
	The operators $\mathbf{T}^{t_2}_{t_1}$ and $\breve{\mathbf{T}}^{t_2}_{t_1}$ are called the {\em negative} and {\em positive type Lax-Oleinik operators}, respectively, and $\mathbf{T}^{t_2}_{t_1}\phi$ and $\breve{\mathbf{T}}^{t_2}_{t_1}\phi$ are called the {\em negative} and {\em positive type Lax-Oleinik evolution of $\phi$}, respectively.
\end{defn}

\begin{defn}\label{defn:Lip_in_large}
Let $(x,d)$ be a metric space. A function $\phi:X\to\R$ is called $(\kappa_1,\kappa_2)$-Lipschitz in the large if there exists $\kappa_1,\kappa_2\geqslant 0$ such that
\begin{align*}
	|\phi(y)-\phi(x)|\leqslant\kappa_1+\kappa_2d(x,y),\quad\forall x,y\in X.
\end{align*}	
\end{defn}

\begin{ex}\label{example:Lip_in_large}
Given $\phi:X\to\R$. We have that
\begin{enumerate}[(i)]
	\item If $X$ is compact, it is obvious that $\phi$ is $(\kappa_1,\kappa_2)$-Lipschitz in the large if and only if $\phi$ is bounded.
	\item If $X=\R^n$ or any complete Riemannian manifold and $\phi$ is uniformly continuous, then for any $\varepsilon>0$ there exists $K_{\varepsilon}>0$ such that $\phi$ is $(\varepsilon,K_{\varepsilon})$-Lipschitz in the large (see Proposition \ref{UC_Lip}).
	\item If $\phi$ is Lipschitz with constant $\mbox{\rm Lip}\,(\phi)$, then $\phi$ is $(0,\mbox{\rm Lip}\,(\phi))$-Lipschitz in the large.
\end{enumerate}
\end{ex}

\begin{Rem}
We have some remarks on the operators $\mathbf{T}^{t_2}_{t_1}\phi$ and $\breve{\mathbf{T}}^{t_2}_{t_1}\phi$.
\begin{itemize}[--]
  \item Notice that there is no extra assumption on the function $\phi$ in Definition \ref{defn:Lax_Oleinik}. But, to ensure that $\mathbf{T}^{t_2}_{t_1}\phi$ and $\breve{\mathbf{T}}^{t_2}_{t_1}\phi$ are finite-valued and the infimum and supremum in Definition \ref{defn:Lax_Oleinik} can be achieved, we need more conditions.
  \item In \cite{Bernard2012}, the author pointed out that if $\phi$ is {\em continuous and Lipschitz in the large}, then $u(t,x)=(\mathbf{T}^t_0\phi)(x)$ is finite-valued for any classical time-dependent Lagrangian $L(t,x,v)$. For more informations on functions that are Lipschitz in the large and applications to Lax-Oleinik evolution in classical case, see \cite{Fathi2019,Cannarsa_Cheng_Fathi2018}.
  \item Using an idea from the proof of Lemma 3.1 in \cite{Cannarsa_Cheng3} (see also \cite{Zhao_Cheng2018} when the Lagrangian has the form $L(x,v,r)$), we can show that the infimum and supremum in Definition \ref{defn:Lax_Oleinik} can be achieved if $\phi$ is lower and upper semi-continuous respectively, and $(\kappa_1,\kappa_2)$-Lipschitz in the large. See Lemma \ref{inf_min} below.
  \item Moreover, if $\phi$ is lower and upper semi-continuous respectively, and $(\kappa_1,\kappa_2)$-Lipschitz in the large, then $\mathbf{T}^{t_2}_{t_1}\phi$ and $\breve{\mathbf{T}}^{t_2}_{t_1}\phi$ satisfies the following {\em Markov property}:
  \begin{align*}
  	\mathbf{T}^{t_3}_{t_2}\circ\mathbf{T}^{t_2}_{t_1}=\mathbf{T}^{t_3}_{t_1},\quad \breve{\mathbf{T}}^{t_3}_{t_2}\circ\breve{\mathbf{T}}^{t_2}_{t_1}=\breve{\mathbf{T}}^{t_3}_{t_1},
  \end{align*}
  whenever $t_1<t_2<t_3$. We can also have that $\lim_{t\to0^+}\mathbf{T}^t_0\phi=\phi$ and $\lim_{t\to0^+}\breve{\mathbf{T}}^t_0\phi=\phi$ if $\phi$ is lower and upper semi-continuous respectively, and $(\varepsilon,K_{\varepsilon})$-Lipschitz in the large for any $\varepsilon>0$. Therefore, it is natural to set both $\mathbf{T}^t_t$ and $\breve{\mathbf{T}}^t_t$ ($t\geqslant0$) to be the identity.
  \item It is useful to regard the definition of $\mathbf{T}^{t_2}_{t_1}$ or $\breve{\mathbf{T}}^{t_2}_{t_1}$ as a representation of {\em marginal functions}. More precisely, set $F_{\phi}(t_1,t_2,\cdot,x,\phi(\cdot))=\phi(\cdot)+h_L(t_1,t_2,y,x,\phi(\cdot))$, If the infimum in the definition of $(\mathbf{T}^{t_2}_{t_1})\phi$ can be achieved in a compact subset $S\subset\R^n$, i.e.,
  \begin{align*}
  	(\mathbf{T}^{t_2}_{t_1})\phi(x)=\inf_{y\in S}F_{\phi}(t_1,t_2,y,x,\phi(y))
  \end{align*}
  then the Lipschitz and semiconcavity estimates can be obtained directly from the {\em uniform} Lipschitz and semiconcavity estimates for $h_L$ (see, for instance, \cite[Theorem 3.4.4]{Cannarsa_Sinestrari_book}). This is also a key point of our program for the study of the propagation of singularities of viscosity solutions (see, for instance, \cite{Cannarsa_Cheng3,CCMW2018,Cannarsa_Cheng_Fathi2017,Cannarsa_Chen_Cheng2019}).
\end{itemize}
\end{Rem}


\begin{Lem}\label{inf_min}
Let $t_2>t_1$ and $x\in\R^n$. If the function $\phi:\R^n\to\R$ is lower semi-continuous and $(\kappa_1,\kappa_2)$-Lipschitz in the large, then there exists $y\in\R^n$ such that $(\mathbf{T}^{t_2}_{t_1}\phi)(x)=\phi(y)+h_L(t_1,t_2,y,x,\phi(y))$. Moreover, for such a minimizer $y$ we have
\begin{equation}\label{eq:r_min}
	|y-x|\leqslant\kappa_1+\{c_0+\overline{\theta}_0(0)+\theta^*(\kappa_2+e^{2K(t_2-t_1)})+|\phi(x)|C\}(t_2-t_1)
\end{equation}
where $C=\sup_{t>0}(1-e^{-2Kt})/t$.
\end{Lem}

\begin{proof}
Fix $t_2>t_1$ and $x\in\R^n$. For any $y\in\R^n$, let $\xi_y$ be a minimizer for $h_L(t_1,t_2,y,x,\phi(y))$ and $u_{\xi_y}$ is determined by
\begin{align*}
	\begin{cases}
		\dot{u}_{\xi_y}(s)=L(s,\xi_y(s),\dot{\xi}_y(s),u_{\xi_y}(s)),\quad s\in[t_1,t_2],&\\
		u_{\xi_y}(t_1)=\phi(y).&
	\end{cases}
\end{align*}
It follows that
\begin{equation}\label{eq:Cara_ODE}
	\dot{u}_{\xi_y}(s)=L(s,\xi_y(s),\dot{\xi}_y(s),0)+\widehat{L_u}(s)\cdot u_{\xi_y}(s),
\end{equation}
where $\widehat{L_u}(s)=\int^1_0L_u(s,\xi_y(s),\dot{\xi}_y(s),\lambda u_{\xi_y}(s))\ d\lambda$. Solving \eqref{eq:Cara_ODE}, we obtain that
\begin{equation}\label{eq:solving1}
	\begin{split}
		u_{\xi_y}(t_2)=&\,e^{\int^{t_2}_{t_1}\widehat{L_u}\ ds}\phi(y)+e^{\int^{t_2}_{t_1}\widehat{L_u}\ ds}\int^{t_2}_{t_1}e^{-\int^s_{t_1}\widehat{L_u}\ d\tau}L(s,\xi_y(s),\dot{\xi}_y(s),0)\ ds\\
		\geqslant&\,e^{\int^{t_2}_{t_1}\widehat{L_u}\ ds}\phi(y)+\int^{t_2}_{t_1}e^{\int^{t_2}_s\widehat{L_u}\ d\tau}(\theta_0(|\dot{\xi}_y(s)|)-c_0)\ ds\\
		\geqslant&\,e^{\int^{t_2}_{t_1}\widehat{L_u}\ ds}\phi(y)+e^{-K(t_2-t_1)}\int^{t_2}_{t_1}\theta_0(|\dot{\xi}_y(s)|)\ ds-c_0(t_2-t_1)e^{K(t_2-t_1)}.
	\end{split}
\end{equation}
Let $\eta(s)\equiv x$ for $s\in[0,t]$ and $u_{\eta}$ satisfies
\begin{align*}
	\begin{cases}
		\dot{u}_{\eta}(s)=L(s,\eta(s),\dot{\eta}(s),u_{\eta}(s))=L(s,x,0,u_{\eta}(s)),\quad s\in[t_1,t_2],&\\
		u_{\eta}(t_1)=\phi(x).&
	\end{cases}
\end{align*}
Similarly, we have that
\begin{equation}\label{eq:solving2}
	\begin{split}
		u_{\eta}(t_2)=&\,e^{\int^{t_2}_{t_1}\widetilde{L_u}\ ds}\phi(x)+e^{\int^{t_2}_{t_1}\widetilde{L_u}\ ds}\int^{t_2}_{t_1}e^{-\int^s_{t_1}\widetilde{L_u}\ d\tau}L(s,x,0,0)\ ds\\
		\leqslant&\,e^{\int^{t_2}_{t_1}\widetilde{L_u}\ ds}\phi(x)+\int^{t_2}_{t_1}e^{\int^{t_2}_s\widetilde{L_u}\ d\tau}\overline{\theta}_0(0)\ ds\\
		\leqslant&\,e^{\int^{t_2}_{t_1}\widetilde{L_u}\ ds}\phi(x)+\overline{\theta}_0(0)(t_2-t_1)e^{K(t_2-t_1)},
	\end{split}
\end{equation}
where $\widetilde{L_u}(s)=\int^1_0L_u(s,x,0,\lambda u_{\eta}(s))\ d\lambda$. Combining \eqref{eq:solving1} and \eqref{eq:solving2} we obtain that
\begin{align*}
	&\,(\phi(y)+h_L(t_1,t_2,y,x,\phi(y)))-(\phi(x)+h_L(t_1,t_2,x,x,\phi(x)))=u_{\xi_y}(t_2)-u_{\eta}(t_2)\\
	\geqslant&\,e^{\int^{t_2}_{t_1}\widehat{L_u}\ ds}\phi(y)-e^{\int^{t_2}_{t_1}\widetilde{L_u}\ ds}\phi(x)+e^{-K(t_2-t_1)}\int^{t_2}_{t_1}\theta_0(|\dot{\xi}_y(s)|)\ ds\\
	&\,-(c_0+\overline{\theta}_0(0))(t_2-t_1)e^{K(t_2-t_1)}\\
	\geqslant&\,-e^{K(t_2-t_1)}|\phi(y)-\phi(x)|-(e^{K(t_2-t_1)}-e^{-K(t_2-t_1)})|\phi(x)|\\
	&\,+e^{-K(t_2-t_1)}\int^{t_2}_{t_1}\theta_0(|\dot{\xi}_y(s)|)\ ds-(c_0+\overline{\theta}_0(0))(t_2-t_1)e^{K(t_2-t_1)}.
\end{align*}
Set $\Lambda_x=\{y\in\R^n: (\phi(y)+h_L(t_1,t_2,y,x,\phi(y)))-(\phi(x)+h_L(t_1,t_2,x,x,\phi(x)))\leqslant0\}$. Notice $\Lambda_x$ is closed since $\phi$ is lower semi-continuous. Recalling that $\phi$ is $(\kappa_1,\kappa_2)$-Lipschitz in the large, for any $y\in\Lambda_x$ and any $a>0$ we have that
\begin{align*}
	0\geqslant&\,-e^{-K(t_2-t_1)}(\kappa_1+\kappa_2|y-x|)-(e^{K(t_2-t_1)}-e^{-K(t_2-t_1)})|\phi(x)|\\
	&\,+e^{-K(t_2-t_1)}\inf_{\xi\in\Gamma^{t_1,t_2}_{y,x}}\int^{t_2}_{t_1}\theta_0(|\dot{\xi}(s)|)\ ds-(c_0+\overline{\theta}_0(0))(t_2-t_1)e^{K(t_2-t_1)}\\
	\geqslant&\,-e^{-K(t_2-t_1)}(\kappa_1+\kappa_2|y-x|)-(e^{K(t_2-t_1)}-e^{-K(t_2-t_1)})|\phi(x)|\\
	&\,+e^{-K(t_2-t_1)}\inf_{\xi\in\Gamma^{t_1,t_2}_{y,x}}\int^{t_2}_{t_1}(a|\dot{\xi}(s)|-\theta^*(a))\ ds-(c_0+\overline{\theta}_0(0))(t_2-t_1)e^{K(t_2-t_1)}\\
	\geqslant&\,-e^{-K(t_2-t_1)}(\kappa_1+\kappa_2|y-x|)-(e^{K(t_2-t_1)}-e^{-K(t_2-t_1)})|\phi(x)|+ae^{-K(t_2-t_1)}|y-x|\\
	&\,-(c_0+\overline{\theta}_0(0))(t_2-t_1)e^{K(t_2-t_1)}-\theta^*(a)(t_2-t_1)e^{-K(t_2-t_1)}.
\end{align*}
It follows
\begin{align*}
	&\,e^{-2K(t_2-t_1)}(a-\kappa_2)|y-x|\\
	\leqslant&\,\kappa_1e^{-2K(t_2-t_1)}+|\phi(x)|C(t_2-t_1)+(c_0+\overline{\theta}_0(0)+\theta^*(a))(t_2-t_1).
\end{align*}
Taking $a=\kappa_2+e^{2K(t_2-t_1)}$, then \eqref{eq:r_min} follows. Thus the set $\Lambda_x$ is compact and the proof is complete.
\end{proof}

\subsection{Representation formula}

In this section, we want to give a representation formula for the viscosity solution of \eqref{eq:HJe_t} in the form of Lax-Oleinik evolution $u(t,x)$ defined as follows: for any $\phi:\R^n\to[-\infty,+\infty]$, set
\begin{equation}\label{eq:value_function}
\begin{split}
	u(t,x)=&\,(\mathbf{T}^{t}_{0}\phi)(x)=\inf_{y\in\R^n}\{\phi(y)+h_L(0,t,y,x,\phi(y))\}\\
	=&\,\inf_{\xi}\left\{\int^t_0L(s,\xi,\dot{\xi},u_{\xi})\ ds+\phi(\xi(0))\right\},
\end{split}
\end{equation}
where the infimum is taken over the set
\begin{align*}
	\mathcal{A}_{t,x}=\{\xi\in W^{1,1}([0,t],\R^n): \xi(t)=x\},
\end{align*}
and $u_\xi$ satisfies the Carath\'eodory equation
\begin{equation}\label{eq:Cara_1}
	\begin{cases}
		\dot{u}_{\xi}(s)=L(s,\xi(s),\dot{\xi}(s),u_{\xi}(s)),\quad a.e.\ s\in[0,t],&\\
		u_{\xi}(0)=\phi(\xi(0)).&
	\end{cases}
\end{equation}

The following principle of dynamic programming is analogous to the classical one.

\begin{Pro}[dynamic programming]\label{dynamic_programming_principle}
Let $(t,x)\in(0,+\infty)\times\R^n$ and $\xi\in\mathcal{A}_{t,x}$. Then for any $t'\in[0,t]$ we have that
\begin{equation}\label{eq:dynamical_programming}
	u(t,x)\leqslant\int^t_{t'}L(s,\xi(s),\dot{\xi}(s),u_{\xi}(s))\ ds+u(t',\xi(t')),
\end{equation}
where $u_{\xi}$ is determined by
\begin{equation}\label{eq:Cara_2}
	\begin{cases}
		\dot{u}_{\xi}(s)=L(s,\xi(s),\dot{\xi}(s),u_{\xi}(s)),\quad a.e.\ s\in[t',t],&\\
		u_{\xi}(t')=u(t',\xi(t')).&
	\end{cases}
\end{equation}
In addition, $\xi\in\mathcal{A}_{t,x}$ is a minimizer for \eqref{eq:value_function} if and only if the equality holds in \eqref{eq:dynamical_programming} for all $t'\in[0.t]$.
\end{Pro}

\begin{proof}
Fix $t>0$ and $x\in\R^n$. Let $t'\in[0,t]$ and $\eta:[0,t]\to\R^n$ be any absolutely continuous function on $[0.t']$ such that $\eta(t')=\xi(t')$. Set
\begin{align*}
\gamma(s)=
\begin{cases}
	\eta(s),& s\in[0,t'];\\
		\xi(s),& s\in[t',t],
\end{cases}
\end{align*}
and
\begin{align*}
	\begin{cases}
		\dot{u}_{\gamma}(s)=L(s,\gamma(s),\dot{\gamma}(s),u_{\gamma}(s)),\quad a.e.\ s\in[0,t],&\\
		u_{\gamma}(0)=\phi(\gamma(0))=\phi(\eta(0)).&
	\end{cases}
\end{align*}
It follows that
\begin{align*}
	u_{\gamma}(t')=&\,\int^{t'}_0L(s,\gamma,\dot{\gamma},u_{\gamma})\ ds+\phi(\eta(0))\\
	u_{\gamma}(t)=&\,\int^{t}_{t'}L(s,\gamma,\dot{\gamma},u_{\gamma})\ ds+u_{\gamma}(t')\\
	=&\,\int^{t}_{t'}L(s,\gamma,\dot{\gamma},u_{\gamma})\ ds+\int^{t'}_0L(s,\gamma,\dot{\gamma},u_{\gamma})\ ds+\phi(\eta(0))
\end{align*}
Therefore
\begin{align*}
	u(t,x)\leqslant&\,u_{\gamma}(t)\leqslant\int^t_{t'}L(s,\xi,\dot{\xi},u_{\xi})\ ds+\int^{t'}_0L(s,\eta,\dot{\eta},u_{\eta})\ ds+\phi(\eta(0)),
\end{align*}
where $u_{\xi}$ and $u_{\eta}$ are the restriction of $u_{\gamma}$ on $[t,t']$ and $[0.t']$ respectively. Taking the infimum over all $\eta$ and recalling that $\xi(t')=\eta(t')$ we obtain \eqref{eq:dynamical_programming}

Now we turn to the proof of the last assertion. If the equality holds in \eqref{eq:dynamical_programming} for all $t'\in[0,t]$, then choosing $t'=0$ yielding that $\xi$ is a minimizer for \eqref{eq:value_function}. Conversely, if $\xi$ is a minimizer for \eqref{eq:value_function}, by \eqref{eq:dynamical_programming} we obtain that for all $t'\in[0,t]$
\begin{equation}\label{eq:dp}
	\begin{split}
		\int^t_0L(s,\xi,\dot{\xi},u_{\xi})\ ds+\phi(\xi(0))=&\,u(t,x)\\
	\leqslant&\,\int^t_{t'}L(s,\xi,\dot{\xi},u_{\xi})\ ds+u(t',\xi(t')),
	\end{split}
\end{equation}
where $u_{\xi}$ is determined by \eqref{eq:Cara_2}. Invoking the definition of $u(t',\xi(t'))$, this implies the inequality in \eqref{eq:dp} is indeed an equality. It follows that the restriction of $\xi$ on $[0,t']$ is a minimizer for $u(t',\xi(t'))$.
\end{proof}

\begin{Pro}
Let $\phi:\R^n\to\R$ be lower semi-continuous and $(\kappa_1,\kappa_2)$-Lipschitz in the large, and let $t>0$. Then the following holds true.
\begin{enumerate}[\rm (1)]
  \item $u(t,x)$ is finite-valued for all $t>0$ and $x\in\R^n$. Moreover the infimum in the definition of $u(t,x)$ is achieved by some $y_{t,x}\in\R^n$.
  \item Suppose that for any $\varepsilon>0$ there exists $K_{\varepsilon}>0$ such that $\phi$ is $(\varepsilon,K_{\varepsilon})$-Lipschitz in the large\footnote{As mentioned in Example \ref{example:Lip_in_large}, a uniformly continuous function on $\R^n$ is $(\varepsilon,K_{\varepsilon})$ Lipschitz in the large.}. Then $\lim_{t\to0^+}|y_{t,x}-x|=0$.
  \item If $\phi$ is bounded and Lipschitz with constant $\mbox{\rm Lip}\,(\phi)$, then there exists $\mu(t)>0$ such that $|y_{t,x}-x|\leqslant\mu(t)t$ for all $t>0$. Moreover, one can take $\mu(t)=c_0+\overline{\theta}_0(0)+\theta^*(\mbox{\rm Lip}\,(\phi)+e^{2Kt})+C\|\phi\|_{\infty}$ for some constant $C>0$.
\end{enumerate}
\end{Pro}

\begin{proof}
Assertion (1) is a reformulation of Lemma \ref{inf_min}. For the proof of (2), set $r_{\varepsilon}=\overline{c}_0+\overline{\theta}_0(0)+\overline{c}_1+\theta^*(K_{\varepsilon}+e^{2Kt})+|\phi(x)|C$. By Lemma \ref{inf_min} we conclude
\begin{align*}
	|y_{t,x}-x|\leqslant\varepsilon+r_{\varepsilon}t.
\end{align*}
This implies $\lim_{t\to0^+}|y_{t,x}-x|=0$. The last assertion (3) is obvious since $\phi$ is  $(0,\mbox{\rm Lip}\,(\phi))$-Lipschitz in the large.
\end{proof}

\begin{Pro}
If $\phi$ is lower semi-continuous and $(\kappa_1,\kappa_2)$-Lipschitz in the large, then $u(t,x)$ defined in \eqref{eq:value_function} is a viscosity solution of \eqref{eq:HJe_t}. 	
\end{Pro}

\begin{Rem}
Uniqueness results for \eqref{eq:HJe_t} hold under further regularity assumptions. See, for instance, \cite[Theorem 5.2]{Barles2013}).
\end{Rem}

\begin{proof}
Fix $t>0$ and $x\in\R^n$. Suppose that $\varphi$ is a $C^1$-function on $(0,\infty)\times\R^n$ such that $u-\varphi$ attains a local maximum at $(t,x)\in U$, a neighborhood of $(t,x)$ in $(0,\infty)\times\R^n$. For any $(t',x')\in U$ ($t'<t$) and any $C^1$ curve $\xi\in\Gamma^{t',t}_{x',x}$, we conclude that
\begin{align*}
	u(t',\xi(t'))-\varphi(t',\xi(t'))\leqslant u(t,\xi(t))-\varphi(t,\xi(t)).
\end{align*}
Invoking dynamic programming principle (Proposition \ref{dynamic_programming_principle}) we obtain that
\begin{align*}
	\frac{\varphi(t,\xi(t))-\varphi(t',\xi(t'))}{t-t'}\leqslant\frac{u(t,x)-u(t',x')}{t-t'}\leqslant\frac 1{t-t'}\int^{t}_{t'}L(s,\xi,\dot{\xi},u_{\xi})\ ds
\end{align*}
where $u_{\xi}$ is determined by
\begin{align*}
	\begin{cases}
		\dot{u}_{\xi}(s)=L(s,\xi(s),\dot{\xi}(s),u_{\xi}(s)),\quad a.e.\ s\in[t',t],&\\
		u_{\xi}(t')=u(t',\xi(t')).&
	\end{cases}
\end{align*}
Taking the limit as $t'\to t$,
\begin{align*}
	D_t\varphi(t,x)+D_x\varphi(t,x)\cdot\dot{\xi}(t)-L(t,x,\dot{\xi}(t),u(t,x))\leqslant 0.
\end{align*}
Since $\xi$ is arbitrary, we conclude
\begin{align*}
	D_t\varphi(t,x)+H(t,x,D_x\varphi(t,x),u(t,x))\leqslant0.
\end{align*}
This implies $u$ is viscosity subsolution of \eqref{eq:HJe_t}.

On the other hand, since $\phi$ is lower semi-continuous and $(\kappa_1,\kappa_2)$-Lipschitz in the large, by Proposition \ref{inf_min}, there exists $y\in\R^n$ such that $u(t,x)=\phi(y)+h_L(0,t,y,x,\phi(y))$. Equivalently, there exists a $C^2$ curve $\xi:[0,t]\to\R^n$, $\xi(t)=x$, such that
\begin{align*}
	u(t,x)=\phi(\xi(0))+\int^t_0L(s,\xi(s),\dot{\xi}(s),u_{\xi}(s))\ ds.
\end{align*}
By the dynamic programming principle, we conclude that
\begin{align*}
	u(t,x)=\int^t_{t'}L(s,\xi(s),\dot{\xi}(s),u_{\xi}(s))\ ds+u(t',\xi(t')),\quad\forall t'\in[0,t].
\end{align*}
In a similar way, one can show that $u$ is viscosity supersolution of \eqref{eq:HJe_t}. This completes the proof.
\end{proof}

\appendix

\section{Some facts from analysis and differential equations}

\subsection{Carath\'eodory equations}

Let $\Omega\subset\R^{n+1}$ be an open set. A function $f:\Omega\subset\R\times\R^n\to\R^n$ is said to satisfy {\em Carath\'eodory condition} if
\begin{itemize}[-]
  \item for any $x\in\R^n$, $f(\cdot,x)$ is measurable;
  \item for any $t\in\R$, $f(t,\cdot)$ is continuous;
  \item for each compact subset $U$ of $\Omega$, there is an integrable function $m_U(t)$ such that
  $$
  |f(t,x)|\leqslant m_U(t),\quad (t,x)\in U.
  $$
\end{itemize}
A classical problem is to find an absolutely continuous function $x$ defined on a real interval $I$ such that $(t,x(t))\in\Omega$ for $t\in I$ and satisfies the following Carath\'eodory equation
\begin{equation}\label{eq:caratheodory_system}
	\dot{x}(t)=f(t,x(t)),\quad a.e., t\in I.
\end{equation}

\begin{Pro}[Carath\'eodory]\label{caratheodory}
	If $\Omega$ is an open set in $\R^{n+1}$ and $f$ satisfies the Carath\'eodory conditions on $\Omega$, then, for any $(t_0, x_0)$ in $\Omega$, there is a solution of \eqref{eq:caratheodory_system} through $(t_0, x_0)$. Moreover, if the function $f(t,x)$ is also locally Lipschitzian in $x$ with a measurable Lipschitz function, then the solution is unique.
\end{Pro}

\noindent For the proof of Proposition \ref{caratheodory} and more results related to Carath\'eodory equation \eqref{eq:caratheodory_system}, the readers can refer to \cite{Coddington_Levinson_book,Filippov_book}.

\subsection{Convexity}

The following facts on the convexity is essentially known (see \cite{CCWY2018}) when the Lagrangian is independent of $t$.

\begin{Lem}\label{convexity}
Let $L$ satisfy conditions \mbox{\rm (L1)-(L3)} and $s\in[a,b]$. We conclude that
\begin{enumerate}[\rm (a)]
  \item The function
  \begin{equation}\label{eq:f_epsilon}
	f(\varepsilon):=L_v(s,x,v/(1+\varepsilon),r)\cdot v/(1+\varepsilon)-L(s,x,v/(1+\varepsilon),r)
  \end{equation}
  is decreasing for $\varepsilon>-1$. In particular,
  \begin{align*}
  	f(\varepsilon)\geqslant f(+\infty)=-L(s,x,0,r)\geqslant-\overline{\theta}_0(0)-K|r|.
  \end{align*}
  \item If $\varepsilon_1,\varepsilon_2>-1$ and $\varepsilon_1<\varepsilon_2$, then we have
  \begin{align*}
  	&\,L(s,x,r,v/(1+\varepsilon_2))\\
  	\leqslant&\,(\kappa+1)^{-1}L(s,x,r,v/(1+\varepsilon_1))+\kappa\cdot(\kappa+1)^{-1}(\overline{\theta}_0(0)+K|r|)
  \end{align*}
  and
  \begin{align*}
  	f(\varepsilon_2)\leqslant\kappa^{-1}L(s,x,r,v/(1+\varepsilon_1))-(\kappa^{-1}+1)L(s,x,r,v/(1+\varepsilon_2))
  \end{align*}
  where $\kappa=(\varepsilon_2-\varepsilon_1)/(1+\varepsilon_1)>0$.
\end{enumerate}
\end{Lem}

\subsection{Uniformly continuous functions}

\begin{Pro}\label{UC_Lip}
	Let $f$ be uniformly continuous function on $\R^n$, then for any $\varepsilon>0$ there exists  $K>0$ such that
	$$
	    |f(x)-f(y)|\leqslant K|x-y|+\varepsilon,\quad \forall x,y\in \R^n.
	$$
\end{Pro}

\begin{proof}
	Suppose that $f$ is uniformly continuous on $\R^n$ and fix $\varepsilon>0$. Then there exists $\delta>0$ such that $|f(z)-f(z')|\leqslant\varepsilon$ whenever $|z-z'|\leqslant\delta$. For any $x,y\in\R^n$, let $\gamma:[0,1]\to\R^n$ be the straight line segment connecting $x$ to $y$, and let $\tau>0$ such that $|x-\gamma(\tau)|=\delta$. Define $z_k=\gamma(k\tau)$, $k=0,\ldots,N$, where $N=[\frac 1{\tau}]$, the integer part of $\frac 1{\tau}$. Then it is clear that
	$$
	|z_k-z_{k+1}|=\delta, \ k=0,\ldots, N-1,\quad\text{and}\quad |z_N-y|\leqslant\delta
	$$
	and
	$$
	|x-y|=\sum^{N-1}_{k=0}|z_k-z_{k+1}|+|z_N-y|\geqslant (N-1)\delta.
	$$
	Therefore,
	\begin{align*}
	|f(x)-f(y)|\leqslant\sum^{N-1}_{k=0}|f(z_{k+1})-f(z_k)|+|f(z_N)-f(y)|\leqslant N\varepsilon\leqslant\frac{\varepsilon}{\delta}|x-y|+\varepsilon.
	\end{align*}
	Picking $K=\frac{\varepsilon}{\delta}$, we complete the proof.
\end{proof}

\subsection{A priori estimates and existence of minimizers}
In this section, fixing real numbers $a<b$, $u\in\R$, $R>0$ and two points $x,y\in\R^n$ such that $|x-y|\leq R$. For convenience, we collect some a priori estimate on the minimizer $\xi$ for \eqref{eq:app_fundamental_solution} and related solution $u_{\xi}$ of \eqref{eq:app_caratheodory_L}. The details of the estimates can be found in Appendix \ref{sec:existence}.

We suppose $\xi$ is a minimizer for \eqref{eq:app_fundamental_solution} and $u_{\xi}$ is determined by \eqref{eq:app_caratheodory_L}.

\begin{Pro}\label{a_priori}
There exists a continuous function $F:[0,+\infty)\times[0,+\infty)\to[0,+\infty)$ depending on $R$ and $u$ continuously, with $F(r_1,\cdot)$ being nondecreasing and superlinear and $F(\cdot,r_2)$ being nondecreasing, such that
\begin{gather*}
	|u_{\xi}(s)|\leqslant F(b-a,R/(b-a)),\quad s\in[a,b],\\
	\int^b_a|L(s,\xi(s),\dot{\xi}(s),u_{\xi}(s))|\ ds\leqslant F(b-a,R/(b-a)),\\
	\operatorname*{ess\ inf}_{s\in[a,b]}|\dot{\xi}(s)|\leqslant F(b-a,R/(b-a)),\\
	\sup_{s\in[a,b]}|\xi(s)-x|\leqslant (b-a)F(b-a,R/(b-a)).
\end{gather*}
\end{Pro}

\section{Existence result and a priori estimates}\label{sec:existence}

In this section, fixing real numbers $a<b$, $u\in\R$, $R>0$ and two points $x,y\in\R^n$ such that $|x-y|\leqslant R$, we shall give some a priori estimate for solutions of the Carath\'eodory equation \eqref{eq:app_caratheodory_L}. Then we show that the action functional $J(\xi)$ defined by \eqref{eq:app_fundamental_solution} attains its minimum on some element in $\Gamma^{a,b}_{x,y}$. For convenience, we set
\begin{align*}
	u_{\xi}(a)=u,\qquad t=b-a.
\end{align*}
Recalling Remark \ref{rem:constants} we can take nonnegative constants $c_0,c_1,K,C_1,C_2$ instead of functions $c_0(\cdot),c_1(\cdot),K(\cdot),C_1(\cdot),C_2(\cdot)$ in our assumptions. For $\varepsilon>0$,
\begin{equation}
\mathcal{A}_{\varepsilon}=\{\xi\in\mathcal{A}:\inf_{\eta\in\mathcal{A}}J(\eta)+\varepsilon \geq u_{\xi}(b)-u\}.
\end{equation}
We denote $L_0(s,x,v)=L(s,x,v,0)$ which is a Lagrangian satisfies the standard conditions in \cite{Fathi_Maderna2007}.

Let $\xi\in\mathcal{A}_{\varepsilon}$ and let $u_{\xi}$ be determined by
\begin{equation}\label{eq:app_Cara}
	\dot{u}_{\xi}(s)=L(s,\xi(s),\dot{\xi}(s),u_{\xi}(s)),\quad s\in[a,b],
\end{equation}
with $u_{\xi}(a)=u$. Solving \eqref{eq:app_Cara} we obtain
\begin{equation}\label{app:cara_2}
	u_{\xi}(s)-u=(e^{\int^s_a\widehat{L_u^{\xi}}\ dr}-1)u+\int^s_ae^{\int^s_\tau\widehat{L_u^{\xi}}\ dr}L_0(\tau,\xi(\tau),\dot{\xi}(\tau))\ d\tau,
\end{equation}
where
\begin{align*}
	\widehat{L_u^{\xi}}(s)=\int^1_0L_u(s,\xi(s),\dot{\xi}(s),\lambda u_{\xi}(s))\ d\lambda.
\end{align*}

\begin{Lem}\label{u_bound}
Let $\xi\in\mathcal{A}_{\varepsilon}$ with $u_{\xi}$ being determined by the associated Carath\'eodory equation \eqref{eq:app_caratheodory_L} and $\varepsilon>0$. Then there exist two continuous functions $F_1, F_2:[0,+\infty)\times[0,+\infty)\to[0,+\infty)$ depending on $u$, with $F_i(r_1,\cdot)$ being nondecreasing and superlinear and $F_i(\cdot,r_2)$ being nondecreasing for any $r_1,r_2\geqslant0$, $i=1,2$, such that
\begin{equation}\label{eq:bound_u}
	\begin{split}
		|u_{\xi}(s)-u|\leqslant&\,tF_1(t,R/t)+2e^{Kt}\varepsilon,\quad s\in[a,b],\\
	\int^b_a|L(\tau,\xi(\tau),\dot{\xi}(\tau),u_{\xi}(\tau))|d\tau\leqslant&\,F_2(t,R/t)+2e^{Kt}(1+Kt)\varepsilon.
	\end{split}
\end{equation}
Moreover, one can take
\begin{align*}
	F_1(r_1,r_2)=&\,3c_{r_1}e^{Kr_1}|u|+2e^{2Kr_1}(\overline{\theta}_0(r_2)+c_0),\\
	F_2(r_1,r_2)=&\,2c_0+(1+Kr_1)F_1(r_1,r_2).
\end{align*}
where $C_t=\sup_{s\in(0,t]}\frac{e^{Ks}-1}{s}<\infty$.
\end{Lem}


\begin{proof}
Let $\xi\in\mathcal{A}$. By \eqref{app:cara_2} and condition (L2) and (L3), we obtain that for all $s\in[a,b]$
\begin{align*}
	u_{\xi}(s)-u\geqslant&\,-(e^{Kt}-1)|u|+\int^s_ae^{\int^s_\tau\widehat{L_u^{\xi}}\ dr}(\theta_0(|\dot{\xi}|)-c_0)\ d\tau\\
	\geqslant&\,-(e^{Kt}-1)|u|-c_0\int^s_ae^{\int^s_\tau\widehat{L_u^{\xi}}\ dr}\ d\tau\\
	\geqslant&\,-(e^{Kt}-1)|u|-c_0te^{Kt}.
\end{align*}
This gives the lower bound of $u_{\xi}$.

Now we turn to the proof of \eqref{eq:bound_u}. Set $\xi_0(s)=x+s(y-x)/t$ for any $s\in[a,b]$. Then $\xi_0\in\mathcal{A}$. By solving the associated Carath\'eodory equation again, we have that
\begin{align*}
	u_{\xi_0}(b)-u=&\,(e^{\int^s_a\widehat{L_u^{\xi_0}}\ dr}-1)u+\int^s_ae^{\int^s_\tau\widehat{L_u^{\xi_0}}\ dr}L_0(\tau,\xi_0(\tau),\dot{\xi}_0(\tau))\ d\tau\\
	\leqslant&\,(e^{Kt}-1)|u|+te^{Kt}\overline{\theta}_0(R/t).
\end{align*}
Now, suppose $\xi\in\mathcal{A}_{\varepsilon}$. Then $u_{\xi}(b)\leqslant u_{\xi_0}(b)+\varepsilon$ and this lead to
\begin{equation}
	u_{\xi}(b)-u\leqslant(e^{Kt}-1)|u|+te^{Kt}\overline{\theta}_0(R/t)+\varepsilon.
\end{equation}
Combining the lower bound of $u_{\xi}$ above we obtain
\begin{equation}\label{eq:bound_u(b)}
	|u_{\xi}(b)-u|\leqslant(e^{Kt}-1)|u|+te^{Kt}(\overline{\theta}_0(R/t)+c_0)+\varepsilon.
\end{equation}
By \eqref{app:cara_2} at $s=b$ we obtain
\begin{align*}
	u_{\xi}(b)-u=(e^{\int^b_a\widehat{L_u^{\xi}}\ dr}-1)u+\int^b_ae^{\int^b_\tau\widehat{L_u^{\xi}}\ dr}L_0(\tau,\xi(\tau),\dot{\xi}(\tau))\ d\tau.
\end{align*}
In view of \eqref{eq:bound_u(b)} we have
\begin{equation}\label{eq:bound_integral1}
	\begin{split}
		&\,\int^b_ae^{\int^b_\tau\widehat{L_u^{\xi}}\ dr}|L_0(\tau,\xi(\tau),\dot{\xi}(\tau))|\ d\tau\leqslant|u_{\xi}(b)-u|+(e^{\int^b_a\widehat{L_u^{\xi}}\ dr}-1)|u|\\
	\leqslant&\,2(e^{Kt}-1)|u|+te^{Kt}(\overline{\theta}_0(R/t)+c_0)+\varepsilon.
	\end{split}
\end{equation}
By solving \eqref{eq:app_caratheodory_L} again we have that for all $s\in[a,b]$
\begin{align*}
	u_{\xi}(b)-u=(e^{\int^b_s\widehat{L_u^{\xi}}\ dr}-1)u_{\xi}(s)+\int^b_se^{\int^b_\tau\widehat{L_u^{\xi}}\ dr}L_0(\tau,\xi(\tau),\dot{\xi}(\tau))\ d\tau.
\end{align*}
Therefore, by \eqref{eq:bound_u(b)} and \eqref{eq:bound_integral1} we conclude that for all $s\in[a,b]$
\begin{align*}
	|u_{\xi}(s)-u_{\xi}(b)|\leqslant&\,(e^{-\int^b_s\widehat{L_u^{\xi}}\ dr}-1)|u_{\xi}(b)|+e^{-\int^b_s\widehat{L_u^{\xi}}\ dr}\cdot\int^b_se^{\int^b_\tau\widehat{L_u^{\xi}}\ dr}|L_0(\tau,\xi(\tau),\dot{\xi}(\tau))|\ d\tau\\
	\leqslant&\,(e^{Kt}-1)(e^{Kt}|u|+te^{Kt}(\overline{\theta}_0(R/t)+c_0)+\varepsilon)\\
	&\,\quad+e^{Kt}\cdot(2(e^{Kt}-1)|u|+te^{Kt}(\overline{\theta}_0(R/t)+c_0)+\varepsilon)\\
	=&\,3(e^{Kt}-1)e^{Kt}|u|+2te^{2Kt}(\overline{\theta}_0(R/t)+c_0)+2e^{Kt}\varepsilon
\end{align*}
which completes the proof of \eqref{eq:bound_u}.

In view of \eqref{eq:bound_u} and condition (L2) and (L3) we have that
\begin{align*}
	&\,\int^b_a|L_0(s,\xi,\dot{\xi})|ds\leqslant\int^b_a(L_0(s,\xi,\dot{\xi})+2c_0)\ ds\\
	\leqslant&\,2c_0t+u_{\xi}(b)-u+K\int^b_a|u_{\xi}|\ ds\\
	\leqslant&\,2c_0t+F_1(t,R/t)+2e^{Kt}\varepsilon+Kt(F_1(t,R/t)+2e^{Kt}\varepsilon)\\
	=&\,2c_0t+(1+Kt)F_1(t,R/t)+2e^{Kt}(1+Kt)\varepsilon.
\end{align*}
This completes the proof of the second inequality in \eqref{eq:bound_u}.
\end{proof}

\begin{Lem}\label{equi_integrable}
Let $\xi\in\mathcal{A}_{\varepsilon}$ and $\varepsilon>0$. Then there exists a continuous function $F:[0,+\infty)\times[0,+\infty)\to[0,+\infty)$ depending on $u$, with $F(r_1,\cdot)$ being nondecreasing and superlinear and $F(\cdot,r_2)$ being nondecreasing for any $r_1,r_2\geqslant0$, such that
\begin{align*}
\int^t_0|\dot{\xi}(s)|\ ds\leqslant tF(t.R/t)+2e^{Kt}(1+tK)\varepsilon.
\end{align*}
Moreover, the family $\{\dot{\xi}\}_{\xi\in\mathcal{A}_{\varepsilon}}$ is equi-integrable.
\end{Lem}

\begin{proof}
By (L2) and (L3), we obtain
\begin{equation}\label{eq:lower1}
	\begin{split}
		&\,u_{\xi}(b)-u\\
		=&\,\int^b_aL(s,\xi(s),\dot{\xi}(s),u_{\xi}(s))\ ds\geqslant\int^b_a\{L_{0}(s,\xi(s),\dot{\xi}(s))-K|u_{\xi}(s)|\}\ ds\\
		\geqslant&\, \int^b_a\{\theta_0(|\dot{\xi}(s)|)-c_0-K|u_{\xi}(s)|\}\ ds\\
		\geqslant&\,\int^b_a\{|\dot{\xi}(s)|-K|u_{\xi}(s)|-(c_0+\theta^*_0(1))\}\ ds.
		\end{split}
\end{equation}
In view of Lemma \ref{u_bound} and \eqref{eq:lower1}, we obtain that
\begin{align*}
	\int^b_a|\dot{\xi}(s)|\ ds\leqslant&\,\int^b_aK|u_{\xi}(s)|\ ds+t(c_0+\theta^*_0(1))+u_{\xi}(b)-u\\
	\leqslant&\,tK(tF_1(t,R/t)+2e^{Kt}\varepsilon)+t(c_0+\theta^*_0(1))\\
	&\,+(tF_1(t,R/t)+2e^{Kt}\varepsilon)\\
    \leqslant&\,tF_2(t.R/t)+2e^{Kt}(1+tK)\varepsilon.
\end{align*}
	
Now we turn to proof of the equi-integrability of the family $\{\dot{\xi}\}_{\xi\in\mathcal{A}_{\varepsilon}}$. Since $\theta_0$ is a superlinear function, for any $\alpha>0$ there exists $C_{\alpha}>0$ such that $r\leqslant\theta_0(r)/\alpha$ for $r>C_{\alpha}$. Thus, for any measurable subset $E\subset[a,b]$, invoking (L2), (L3) and Lemma \ref{u_bound}, we have that
	\begin{align*}
		\int_{E\cap\{|\dot{\xi}|>C_{\alpha}\}}|\dot{\xi}|ds\leqslant&\,\frac 1{\alpha}\int_{E\cap\{|\dot{\xi}|>C_{\alpha}\}}\theta_0(|\dot{\xi}|)ds\leqslant\frac 1{\alpha}\int_{E\cap\{|\dot{\xi}|>C_{\alpha}\}}\{L_0(s,\xi,\dot{\xi})+c_0\}\ ds\\
		\leqslant&\,\frac 1{\alpha}\int_{E\cap\{|\dot{\xi}|>C_{\alpha}\}}\{L(s,\xi,\dot{\xi},u_{\xi})+K|u_{\xi}(s)|+c_0(s)\}\ ds\\
        \leqslant&\,\frac 1{\alpha}\int^{b}_{a}\{L(s,\xi,\dot{\xi},u_{\xi})+K|u_{\xi}(s)|+c_0(s)\}\ ds\\
        \leqslant&\,\frac 1{\alpha}\left\{(u_{\xi}(b)-u)+tK(tF_1(t,R/t)+2e^{Kt}\varepsilon+|u|)+c_0t\right\}\\
        \leqslant&\,\frac 1{\alpha}\left\{((tF_1(t,R/t)+2e^{Kt}\varepsilon))+tK(tF_1(t,R/t)+2e^{Kt}\varepsilon+|u|)+c_0t\right\}\\
		:=&\,\frac 1{\alpha}F_3(\varepsilon,t,R/t)
	\end{align*}
	Therefore, we conclude that
	\begin{align*}
		\int_E|\dot{\xi}|ds\leqslant\int_{E\cap\{|\dot{\xi}|>C_{\alpha}\}}|\dot{\xi}|ds+\int_{E\cap\{|\dot{\xi}|\leqslant C_{\alpha}\}}|\dot{\xi}|ds\leqslant\frac 1{\alpha}F_3(\varepsilon,t,R/t)+|E|C_{\alpha}.
	\end{align*}
	Then, the equi-integrability of the family $\{\dot{\xi}\}_{\xi\in\mathcal{A}_{\varepsilon}}$ follows since the right-hand side can be made arbitrarily small by choosing $\alpha$ large and $|E|$ small, and this proves our claim.
\end{proof}

\begin{Pro}\label{existence}
	The functional
	\begin{align*}
		\mathcal{A}\ni\xi\mapsto J(\xi)=\int^b_aL(s,\xi(s),\dot{\xi}(s),u_{\xi}(s))\ ds,
	\end{align*}
	where $u_{\xi}$ is determined by \eqref{eq:app_caratheodory_L}, admits a minimizer.
\end{Pro}

\begin{Rem}\label{rem:J}
Notice that we can rewrite the functional $J$ as
\begin{equation}\label{eq:J1}
	J(\xi)=(e^{\int^b_a\widehat{L_u^{\xi}}\ dr}-1)u+\int^b_ae^{\int^b_\tau\widehat{L_u^{\xi}}\ dr}L_0(\tau,\xi(\tau),\dot{\xi}(\tau))\ d\tau
\end{equation}
in spirit of \eqref{app:cara_2} and the fact $J(\xi)=u_{\xi}(b)-u$. We set $\mu_{\xi}(s):=e^{\int^b_s\widehat{L_u^{\xi}}\ dr}$. Therefore $J(\xi)=J_1(\xi)+J_2(\xi)$ where
\begin{align*}
	J_1(\xi)=(\mu_{\xi}(a)-1)u,\quad J_2(\xi)=\int^b_a\mu_{\xi}(\tau)L_0(\tau,\xi(\tau),\dot{\xi}(\tau))\ d\tau.
\end{align*}
\end{Rem}

\begin{proof}
	Fix $x,y\in\R^n$, $b>a$ and $u\in\R$. Consider any minimizing sequence $\{\xi_k\}$ for $J$, that is, a sequence such that $J(\xi_k)\to\inf\{J(\xi):\xi\in\mathcal{A}\}$ as $k\to\infty$. We want to show that this sequence admits a cluster point which is the required minimizer. Notice there exists an associated sequence $\{u_{\xi_k}\}$ given by \eqref{eq:app_caratheodory_L} in the definition of $J(\xi_k)$. The idea of the proof is standard but a little bit different from the classical proof of Tonelli's existence theorem.
	
	First, notice that Lemma \ref{equi_integrable} implies that the sequence of derivatives $\{\dot{\xi}_k\}$ is equi-integrable. Since the sequence $\{\dot{\xi}_k\}$ is equi-integrable, by the Dunford-Pettis Theorem there exists a subsequence, which we still denote by $\{\dot{\xi}_k\}$, and a function $\eta^*\in L^1([a,b],\R^n)$ such that $\dot{\xi}_k\rightharpoonup\eta^*$ in the weak-$L^1$ topology. The equi-integrability of $\{\dot{\xi}_k\}$ implies that the sequence $\{\xi_k\}$ is equi-continuous and uniformly bounded. Invoking the Ascoli-Arzela theorem, we can also assume that the sequence $\{\xi_k\}$ converges uniformly to some absolutely continuous function $\xi_{\infty}\in\Gamma^{a,b}_{x,y}$. For any test function $\varphi\in C^1_0([a,b],\R^n)$,
	\begin{align*}
		\int^b_a\varphi\eta^* ds=\lim_{k\to\infty}\int^b_a\varphi\dot{\xi}_k ds=-\lim_{k\to\infty}\int^b_a\dot{\varphi}\xi_k ds=-\int^b_a\dot{\varphi}\xi_{\infty} ds.
	\end{align*}
	By the du Bois-Reymond lemma (see, for instance, \cite[Lemma 6.1.1]{Cannarsa_Sinestrari_book}), we conclude that $\dot{\xi}_{\infty}=\eta^*$ almost everywhere. In View of Remark \ref{rem:J} and condition (L3), we also have that the sequence $\{\mu_{\xi_k}\}$ is bounded and equi-continuous. Therefore, $\mu_{\xi_k}$ converges uniformly to $\mu_{\xi}$ as $k\to\infty$ by taking a subsequence if necessary.
	
	We recall a classical result (see, for instance, \cite[Theorem 3.6]{Buttazzo_Giaquinta_Hildebrandt_book} or \cite[Section 3.4]{Buttazzo_book}) on the sequentially lower semicontinuous property on the functional
	\begin{align*}
		L^1([a,b],\R^m)\times L^1([a,b],\R^n)\ni (\alpha,\beta)\mapsto\mathbf{F}(\alpha,\beta):=\int^b_a\mathbf{L}(\alpha(s),\beta(s))\ ds.
	\end{align*}
	One has that if
	\begin{inparaenum}[(i)]
	\item $\mathbf{L}$ is lower semicontinuous;
	\item $\mathbf{L}(\alpha,\cdot)$ is convex on $\R^n$,
	\end{inparaenum}
	then the functional $\mathbf{F}$ is sequentially lower semicontinuous on the space $L^1([a,b],\R^m)\times L^1([a,b],\R^n)$ endowed with the strong topology on $L^1([a,b],\R^m)$ and the weak topology on $L^1([a,b],\R^n)$.
	
	Now, let
	\begin{align*}
		\mathbf{L}(\mu_{\xi_k}(s),\xi_k(s),\dot{\xi_k}(s)):=\mu_{\xi_k}(s)L_0(s,\xi_k(s),\dot{\xi}_k(s))
	\end{align*}
	with $\alpha_{\xi_k}(s)=(\mu_{\xi_k}(s),\xi_k(s))$ and $\beta_{\xi_k}(s)=\dot{\xi_k}(s)$. Then $J_2$ is lower semi-continuous in the topology mentioned above. The lower semi-continuity of $J_1$ is obvious (in fact, $J_1$ is continuous).  Therefore, $\xi_{\infty}\in\mathcal{A}$ is a minimizer of $J$ and this completes the proof of the existence result.
\end{proof}

\begin{Cor}\label{bound_1}
	There exists  a continuous function $F:[0,+\infty)\times[0,+\infty)\to[0,+\infty)$ depending on $u$, with $F(r_1,r_2)$ nondecreasing in both variables and superlinear with respect to $r_2$, such that every minimizer $\xi\in\mathcal{A}$ for \eqref{eq:app_fundamental_solution} satisfies
	\begin{align*}
		\int^b_a|\dot{\xi}(s)|\ ds\leqslant tF(t,R/t)
	\end{align*}
and
	\begin{align*}
		\operatorname*{ess\ inf}_{s\in[a,b]}|\dot{\xi}(s)|\leqslant F(t,R/t),\quad\sup_{s\in[a,b]}|\xi(s)-x|\leqslant tF(t,R/t).
	\end{align*}
\end{Cor}

\begin{proof}
	The first assertion is a direct consequence of Lemma \ref{equi_integrable}. The last two inequalities follow from the relations
	\begin{align*}
		\operatorname*{ess\ inf}_{s\in[a,b]}|\dot{\xi}(s)|\leqslant\frac 1t\int^b_a|\dot{\xi}(s)|\ ds,\quad\text{and}\quad |\xi(s)-x|\leqslant\int^b_a|\dot{\xi}(s)|\ ds,
	\end{align*}
	together with the first assertion.
\end{proof}

\section{Herglotz' variational principle on manifolds}

In this section, we try to explain, under the assumptions (L1)-(L4), how to move the Herglotz' generalized variational principle to a closed, connected $n$-dimensional smooth manifold $M$ without boundary. We continue to use the notations $u,t,K,c_{0}$ defined in Appendix B.

Once and for all, we fix a auxiliary Riemannian metric $g$ on $M$ and denote $d_{g}$ the distance induced by $g$. First, we notice that conditions (L1)-(L4) can be adapt to $L:\R\times TM\times\R\to\R$, only differences are:
\begin{itemize}[--]
  \item (L1) is restated as $L(t,x,\cdot,r)$ is strictly convex on $T_{x}M$ for any fixed $(t,x,r)$;
  \item the norms on $\R^{n}$ is replaced by $|\cdot|_{g}$ defined by $g$.
\end{itemize}

Let $\{(B_i,\Phi_i)\}$ be a $C^{2}$ atlas for $M$. Assume that $\{B_i\}_{i=1}^N$ is a finite open cover of $M$, where $\Phi_i:B_i\to\mathbb{D}^n$ is a $C^2$-diffeomorphism for each $i\in\{1,\ldots,N\}$ and $\mathbb{D}^{n}$ denotes the $n$-dimensional unit disc. Thus $\Phi_j^{-1}\circ\Phi_i:B_i\cap B_j\to B_i\cap B_j$ is a $C^2$-diffeomorphism for each pair $i,j$. Let $L(t,x,v,r)$ be a Lagrangian that satisfies (L1)-(L3) together with (L4) or (L4'), for fixed $i$, let $B=B_i$ and $\Phi=\Phi_i:B\to\mathbb{D}^{n}$ be a corresponding local coordinate, then
\begin{align*}
	(\Phi,d\Phi):TB\to\mathbb{D}^{n}\times\R^n
\end{align*}
defines a local trivialization of $TB$ and $L_{\Phi}:\R\times\mathbb{D}^{n}\times\R^n\times\R\to\R$ defined as
\begin{align*}
	L_{\Phi}(t,\bar{x},\bar{v},u)=L(t,\Phi^{-1}(\bar{x}),d\Phi^{-1}(\bar{x})\bar{v},u),\quad (\bar{x},\bar{v})\in \mathbb{D}^{n}\times\R^n,\ u\in\R
\end{align*}
is a representation of $L$ in $(B,\Phi)$. By the local representation performed above, Herglotz' generalized variational principle for $L$ restricted on some local chart $(B,\Phi)$ is equivalent to that for $L_{\Phi}$ on $[a,b]\times\mathbb{D}^{n}\times\R^n\times\R\to\R$ if $\Phi$ is a bi-Lipschitz homeomorphism and a $C^2$-diffeomorphism.

From now on, we fix $a,b\in\R$. Let $x,y\in M$ and $u\in\R$, choosing $\xi\in\Gamma^{a,b}_{x,y}(M)$, we consider the Carath\'eodory equation
\begin{equation}\label{eq:app_caratheodory_L2}
	\begin{cases}
		\dot{u}_{\xi}(s)=L(s,\xi(s),\dot{\xi}(s),u_{\xi}(s)),\quad a.e.\ s\in[a,b],&\\
		u_{\xi}(a)=u.&
	\end{cases}
\end{equation}
Similarly, we define the action functional
\begin{equation}\label{eq:app_fundamental_solution2}
	J(\xi):=\int^b_aL(s,\xi(s),\dot{\xi}(s),u_{\xi}(s))\ ds,
\end{equation}
where $\xi\in\Gamma^{a,b}_{x,y}(M)$ and $u_{\xi}$ is defined in \eqref{eq:app_caratheodory_L2}. Our purpose is to minimize $J(\xi)$ over
\begin{align*}
	\mathcal{A}(M)=\{\xi\in\Gamma^{a,b}_{x,y}(M): \text{\eqref{eq:app_caratheodory_L2} admits an absolutely continuous solution $u_{\xi}$}\}.
\end{align*}
Notice that $\mathcal{A}(M)\not=\varnothing$ because it contains all piecewise $C^1$ curves connecting $x$ to $y$.
Moreover, (L2) implies that any $\xi\in\mathcal{A}(M)$ is absolutely continuous, thus has finite length.

For a fixed $\kappa>0$, assume that $y\in B_{\kappa t}(x)$ and that $\eta\in\mathcal{A}(M)$ is a minimizer of the action functional $\eta\mapsto J(\eta)$. It is obvious that the estimates performed on $\R^{n}$ carry over to the manifold case, then there exist constants $C_1(\kappa,a,b)>0$, $C(u,a,b,\kappa)>0$ such that
\begin{equation}\label{bounds}
\eta(s)\in B_{C_1 t}(x)\ \text{ for }\ s\in[a,b],\quad|\dot{\eta}(s)|_{g}\leqslant C_1,\quad\sup_{s\in[a,b]}|u_{\eta}(s)|\leqslant C.
\end{equation}
The second inequality holds since we only use quantitative derivatives like $L_{u},L_{t}$ in the deduction of Erdmann condition and it can be carried over to the manifold case.

To begin the construction, we notice that there is $r>0$ such that for all $x\in M$, the geodesic ball $B_{r}(x)$ is bi-Lipschitz, $C^{2}$ diffeomorphic to $\mathbb{D}^{n}$ (some rescaling of $\Phi=\exp_{x}^{-1}$ shall give this diffeomorphism). We set $\kappa=\frac{\text{diam}(M)}{t}$, $C_1(\kappa,a,b)+1:=C_{2}(\kappa,a,b)$.

\medskip

\noindent\textbf{Local case}: Assume $2C_{2}t<r$ and $x,y\in B_{\frac{r}{2}}(x_{0})$ for some $x_{0}\in M$. By the discussion above, Herglotz' variational principle for $L$ restricted on the local chart $(B_{r}(x_{0}),\Phi)$ is equivalent to that for $L_{\Phi}$, satisfying all aforementioned assumptions, defined on $[a,b]\times\mathbb{D}^{n}\times\R^n\times\R\to\R$. Thus, by denoting
\begin{align*}
	\mathcal{B}(M)=\{\eta\in\mathcal{A}(M): \text{$\eta(s)\in B_{r}(x)$ for all $s\in[a,b]$}\},
\end{align*}
we can claim that
\begin{align*}
	\inf_{\mathcal{A}(M)}J(\xi)=\inf_{\mathcal{B}(M)}J(\xi)
\end{align*}
and they admit the same minimizers: by applying \eqref{bounds} for any minimizer $\eta$,
\begin{align*}
d_{g}(\eta(s),x_{0})\,&\leqslant d_{g}(\eta(s),x)+d_{g}(x,x_{0})\\
\,&\leqslant\int^{b}_{a}|\dot{\eta}(s)|_{g}\,ds+d_{g}(x,x_{0})\leqslant C_{1}\cdot t+\frac{r}{2}<r,
\end{align*}

We could formulate the conclusions from Section 2 and Appendix B into the following
\begin{Pro}\label{Main_Herglotz}
Assume $2C_{2}t<r$ and $x,y\in B_{\frac{r}{2}}(x_{0})$ for some $x_{0}\in M,\,\,\Phi:B:=B_{r}(x_{0})\rightarrow\mathbb{D}^{n}$ is a local chart at $x_{0}$, then
\begin{enumerate}[\rm (a)]
  \item The functional
	\begin{align*}
		\mathcal{A}(M)\ni\xi\mapsto J(\xi)=\int^b_aL(s,\xi(s),\dot{\xi}(s),u_{\xi}(s))\ ds,
	\end{align*}
	where $u_{\xi}$ is determined by \eqref{eq:app_caratheodory_L2} admits a minimizer on $\mathcal{A}(B)$.
	
  \vskip0.2cm
  \item Let $\xi\in\mathcal{A}(B)$ be a minimizer of $J$, then there is a function $F=F_{u,B}:[0,+\infty)\times[0,+\infty)\to[0,+\infty)$, with $F(\cdot,r)$ being nondecreasing for any $r\geqslant0$, such that
	\begin{equation*}
		|u_{\xi}(s)|\leqslant tF(a,b,\kappa)+G(t)|u|:=C(u,a,b,\kappa),\quad s\in[a,b]
	\end{equation*}
	where $G(t)>0$ is also nondecreasing in $t$.

  \vskip0.2cm
  \item Let $\xi\in\mathcal{A}(B)$ be a minimizer of $J$, then there is a function $F=F_{u,B}:[0,+\infty)\times[0,+\infty)\to[0,+\infty)$, with $F(\cdot,r)$ is nondecreasing for any $r\geqslant0$, such that
	\begin{align*}
		\operatorname*{ess\ sup}_{s\in[a,b]}|\dot{\xi}(s)|\leqslant F(a,b,\kappa):=C_{1}(a,b,\kappa).
	\end{align*}
	\item If $L$ is of class $C^2$, then for any minimizer $\xi$ for \eqref{eq:app_fundamental_solution2} we have
	\begin{enumerate}[\rm 1)]
	\item Both $\xi$ and $u_{\xi}$ are of class $C^2$ and $\xi$ satisfies Herglotz equation \eqref{eq:Herglotz} in local charts for all $s\in[a,b]$ where $u_{\xi}$ is the unique solution of \eqref{eq:app_caratheodory_L2};
	\item Let $p(s)=L_v(s,\xi(s),\dot{\xi}(s),u_{\xi}(s))$ be the dual arc, then $p$ is also of class $C^2$ and we conclude that $(\xi,p,u_{\xi})$ satisfies Lie equation \eqref{eq:Lie} in local charts for all $s\in[a,b]$.
\end{enumerate}
\end{enumerate}
\end{Pro}

\medskip

\noindent\textbf{General case}: This is just the standard ``broken geodesic'' argument. Let $\{(B_i,\Phi_i)\}_{i=1}^N$  be an atlas of $M$ such that $B_{i}=B_{\frac r2}(x_{i})$ and $\{x_{i}\}_{1\leqslant i\leqslant N}$ forms a $\frac{r}{2}$-net on $M$. Without loss of generality, we assume that $x\in B_1$ and $y\in B_N$. Let $\xi\in\mathcal{A}(M)$ be a minimizer which is necessarily to be $C_1$-Lipschitz as the \emph{a priori} estimate shown. Then, there exists a partition $a=t_0<t_1<t_2<\cdots<t_{k-1}<t_k=b$ such that $0\leqslant t_{j+1}-t_{j}\leqslant\frac{r}{2C_{2}}$ and $z_j=\xi(t_j)$ and $z_{j+1}=\xi(t_{j+1})$ are contained in the same $B_i$.
Thus applying Proposition \ref{Main_Herglotz}, as a minimizer of $J$ on $\Gamma^{t_j,t_{j+1}}_{z_j,z_{j+1}}$, $\xi\vert_{[t_j,t_{j+1}]}$ falls in $B_{r}(x_i)$, which reduce the problem to the local case. For each $j$, we define
\begin{align*}
	h_L^j(t_j,t_{j+1},z_j,z_{j+1},u_j)=\inf_{\xi_j}\int^{t_{j+1}}_{t_j}L(s,\xi_j(s),\dot{\xi}_j(s),u_{\xi_j}(s))\ ds,
\end{align*}
where $\xi_j$ is an absolutely continuous curve constrained in $B_{r}(x_i)$ connecting $z_j$ to $z_{j+1}$ and $u_{\xi_j}$ is uniquely determined by \eqref{eq:app_caratheodory_L} with initial condition $u_j$. Now we consider the problem
\begin{equation}\label{eq:conjunction}
	g(a,b,x,y,u):=\inf\sum^k_{j=1}h_L^j(t_j,t_{j+1},z_j,z_{j+1},u_j),
\end{equation}
where the infimum is taken over partitions $a=t_0<t_1<t_2<\cdots<t_{k-1}<t_k=b$ with $t_{j+1}-t_{j}\in[0,\frac{r}{2C_{2}}]$, $z_j,z_{j+1}\in M$ contained in the some $B_i$ and $u_j\in\R$. Due to Proposition \ref{Main_Herglotz} (b), $\{u_j\}$ can be constrained in a compact subset of $\R$ depending only on $u$, $x,y$ and $t$. Therefore the infimum in \eqref{eq:conjunction} can be attained. Thanks to the local semiconcavity of the fundamental solution $h_L^j$, $h_L^j$ is differentiable at each minimizer which leads to the fact
\begin{align*}
	h_L(a,b,x,y,u)=g(a,b,x,y,u).
\end{align*}

\begin{Pro}\label{Main_Herglotz2}
The conclusion of Proposition \ref{Main_Herglotz} holds for any connected and closed $C^2$ manifold $M$ for all $a<b$.
\end{Pro}

\bibliographystyle{abbrv}
\bibliography{mybib}

\end{document}